\documentclass[a4paper,11pt]{article}
\usepackage[english]{babel}

\usepackage{amssymb,amsmath,amscd,amsfonts,amsthm,bbm,mathrsfs}
\usepackage[colorlinks=true, linkcolor=blue, citecolor=blue]{hyperref}
\usepackage{graphicx} 
\usepackage{tikz,pgfplots}
\usepackage{color} 
\usepackage[textwidth=1.6cm, textsize=footnotesize]{todonotes}  
\DeclareMathOperator{\dom}{dom}
\DeclareMathOperator{\zer}{zer}

\DeclareMathOperator{\epi}{epi}
\DeclareMathOperator{\graph}{gr}
\DeclareMathOperator{\prox}{prox}

\DeclareMathOperator{\support}{supp}
\DeclareMathOperator{\proj}{proj}
\DeclareMathOperator{\dist}{dist}
\DeclareMathOperator{\clos}{cl}

\newcommand{\NN}{\mathbb N}
\newcommand{\RR}{\mathbb R}

\newcommand{\PP}{\mathbb P}
\newcommand{\EE}{\mathbb E}

\newcommand{\mA}{{\mathcal A}} 
\newcommand{\mB}{{\mathcal B}} 
\newcommand{\mD}{{\mathcal D}} 
\newcommand{\mZ}{{\mathcal Z}}

\newcommand{\sA}{{\mathsf A}}
\newcommand{\sB}{{\mathsf B}}
\newcommand{\sX}{{\mathsf X}}

\newcommand{\mcB}{{\mathscr B}} 
\newcommand{\mcT}{{\mathscr T}} 
\newcommand{\mcF}{{\mathscr F}} 
\newcommand{\mcX}{{\mathscr X}} 

\newcommand{\RN}{{{\mathbb R}^N}} 

\newcommand{\ps}[1]{\langle #1 \rangle}

\newcommand{\bs}{\boldsymbol}
\newcommand{\eqdef}{:=} 
\newtheorem{theorem}{Theorem}[section]
\newtheorem{lemma}{Lemma}[section]
\newtheorem{corollary}{Corollary}[section]

\newtheorem{proposition}{Proposition}[section]

\theoremstyle{remark}
\newtheorem{remark}{Remark}

\begin{document}

\title{Dynamical Behavior of a Stochastic Forward-Backward Algorithm Using 
Random Monotone Operators 
}

% \subtitle{Do you have a subtitle?\\ If so, write it here}

% if too long for running head
% \titlerunning{Dynamical Behavior of a Stochastic Forward-Backward Algorithm} 

\author{Pascal Bianchi
 \thanks{
   LTCI, CNRS, T\'el\'ecom ParisTech, Universit\'e Paris-Saclay. 
e-mail: \texttt{pascal.bianchi@telecom-paristech.fr}}
 \and Walid Hachem\thanks{LTCI, CNRS, T\'el\'ecom ParisTech, Universit\'e Paris-Saclay. 
e-mail: \texttt{walid.hachem@telecom-paristech.fr}}
}

% \date{\today}

\maketitle

\begin{abstract}
The purpose of this paper is to study the dynamical behavior of the sequence
produced by a forward-backward algorithm, involving two random maximal monotone
operators and a sequence of decreasing step sizes.  Defining a mean monotone
operator as an Aumann integral, and assuming that the sum of the two mean
operators is maximal (sufficient maximality conditions are provided), it is
shown that with probability one, the interpolated process obtained from the
iterates is an asymptotic pseudo trajectory in the sense of Bena\"{\i}m and
Hirsch of the differential inclusion involving the sum of the mean operators.
The convergence of the empirical means of the iterates towards a zero of the
sum of the mean operators is shown, as well as the convergence of the sequence
itself to such a zero under a demipositivity assumption.  These results find
applications in a wide range of optimization problems or variational 
inequalities in random environments. 
\end{abstract}


\section{Introduction}
\label{sec-intro}

In the fields of convex analysis and monotone operator theory, the
forward-backward splitting algorithm \cite{mer-(livre)79,lio-mer-79}
is one of the most often studied techniques for iteratively finding a zero of a
sum of two maximal monotone operators. This problem finds applications
in convex minimization problems. Indeed, when each of the two maximal
monotone operators coincides with the subdifferential of a proper and
lower semicontinuous convex function, the forward-backward algorithm
converges to a minimizer of the sum of the two functions, provided some
conditions are met.  Other applications include saddle point problems and
variational inequalities.  Each iteration of the algorithm involves a
forward step, where one of the operators is used explicitly, followed
a backward step that consists in applying the \emph{resolvent} of the second 
operator to the output of the forward step. 

The purpose of this paper is to study a version of the
forward-backward algorithm, where at each iteration, each of the two
operators is replaced with an operator that has been \emph{randomly chosen}
amongst a collection of maximal monotone operators.  The sequence of
random monotone operators is assumed to be independent and identically
distributed (in a sense that will be made clear below), and the step size of the
algorithm is supposed to approach zero as the number of iterations goes to 
infinity, in order to alleviate the noise effect due to the
randomness.

The aim is to study the dynamical behavior of the stochastic sequence generated
by the above algorithm. Our main result states that the piecewise linear
interpolation of the output sequence is an \emph{asymptotic pseudotrajectory}
(APT) \cite{ben-hir-96,ben-hof-sor-05} of a certain \emph{semiflow}, which we
shall characterize below.  Loosely speaking, it means that the iterates of our
stochastic forward-backward algorithm asymptotically ``shadow'' the trajectory
of a continuous time dynamical system, hence inheriting its convergence
properties. In our case, the latter dynamical system is taken as a differential
inclusion involving the sum of the \emph{Aumann expectations} of the randomly
chosen maximal monotone operators \cite{aum-65,aub-fra-90}, as also introduced
in the recent paper \cite{bia-(arxiv)15}.

The convergence of the algorithm towards an element of the set of zeros of the
sum of the Aumann expectations is of obvious interest. In this regard, the
above APT property yields two important corollaries.  Using a result of
% Bena{\"\i}m and Schreiber 
\cite{ben-sch-00}, we show that the sequence of empirical means of the iterates
converges almost surely (a.s.) to a (random) element of the set of zeros.
Moreover, when the sum of the Aumann expectations is assumed
\emph{demipositive} \cite{bru-75}, we prove that the sequence of iterates
converges a.s. to a zero.  Verifiable conditions for demipositivity can be
easily devised.

This paper is organized as follows.  Section~\ref{sec-pb} provides the
theoretical background.  Section~\ref{sec:results} introduces the main
algorithm and states the main results.  Section~\ref{sec-appli} reviews some
applications to convex minimization problems.  Related works are discussed in
Section~\ref{sec-stArt}. Proofs are provided in Section~\ref{sec-prf}.
Perspectives and conclusions are addressed in Sections~\ref{sec-persp}
and~\ref{sec-conclu} respectively.

\section{Preliminaries}
\label{sec-pb} 

\subsection{Monotone Operators} 
\label{mon} 

A set-valued operator $\sA:\RN \rightrightarrows \RN$, where $N$ is some
positive integer, is said to be monotone if $\forall (x,y)\in \graph(\sA)$,
$\forall (x',y')\in \graph(\sA)$, $\ps{y-y',x-x'}\geq 0$, where $\graph(\sA)$
stands for the graph of $\sA$.  A non-empty monotone operator is said to be
maximal if its graph is a maximal element in the inclusion ordering. A typical
maximal monotone operator is the subdifferential of a function belonging to
$\Gamma_0$, the family of proper and lower semicontinuous convex functions on
$\RR^N$. We use $\mathcal M$ to represent the set of maximal monotone operators
on $\RR^N$, and let $\dom(\sA) \eqdef \{ x \in \RR^N \, : \, \sA(x) \neq
\emptyset \}$ be the domain of the operator $\sA$.

Given that $\sA, \sB \in \mathcal M$, where $\sB$ is assumed to be
single-valued and where $\dom(\sB) = \RR^N$, the forward-backward algorithm
reads 
\begin{equation}
\label{fb-deter}
{\mathsf x}_{n+1} = ( I + \gamma \sA )^{-1} ( {\mathsf x}_n - \gamma \sB({\mathsf x}_n) ) , 
\end{equation} 
where $I$ is the identity operator, $\gamma$ is a real positive step, and
$(\cdot)^{-1}$ is the inverse operator defined by the fact that $(x,y) \in
\graph(A^{-1}) \Leftrightarrow (y,x) \in \graph(A)$ for an operator $A$.  The
operator $( I + \gamma \sA )^{-1}$, called the resolvent, is single valued with
the domain $\RR^N$ since  $\sA \in \mathcal M$
\cite{bre-livre73,bau-com-livre11}.  In the special case where $\sA$ is equal
to the subdifferential $\partial f$ of a function $f\in \Gamma_0$, the
resolvent is also refered to as the proximity operator, and we note $\prox_f(x)
= ( I + \partial f )^{-1} (x)$.

We denote the set of zeros of $\sA$ as $Z(\sA) \eqdef \{ x \in
\RR^N \, : \, 0 \in \sA(x) \}$. Assuming that $\sB$ is so-called cocoercive,
and that $\gamma$ satisfies a certain condition, the forward-backward algorithm
is known to converge to an element of $Z(\sA + \sB)$, provided the latter set
is not empty \cite{bau-com-livre11}.

\subsection{Set-Valued Functions and Set-Valued Integrals} 

Let $(\Xi, \mcT, \mu)$ be a probability space, where $\mcT$ is $\mu$-complete.
Consider the space $\RN$ equipped with its Borel field ${\mcB}(\RN)$, and let
$F: \Xi \rightrightarrows \RN$ be a set-valued function  
such that $F(\xi)$ is a closed set for any $\xi \in \Xi$.  The set-valued
function $F$ is said to be \emph{measurable} if 
$\{ \xi \, : \, F(\xi) \cap H \neq
\emptyset \} \in {\mcT}$ for any set $H \in {\mcB}(\RN)$. This is known to be
equivalent to asserting that the domain 
$\dom(F) \eqdef \{ \xi \in \Xi \, : \, F(\xi) \neq \emptyset \}$ 
of $F$ belongs to $\mcT$, and that there exists a sequence of measurable
functions $\varphi_n : \dom(F) \to \RN$ such that $F(\xi) =
\clos(\{\varphi_n(\xi) \})$ for all $\xi \in \dom(F)$ 
\cite[Chap.~3]{cas-val77} \cite{hia-um-77}.  Assume now that $F$ is
measurable and that $\mu(\dom(F)) = 1$.  For $1 \leq p < \infty$, let 
${\mathcal L}^p(\Xi, {\mcT}, \mu; \RN)$ be the Banach space of measurable
functions $\varphi : \Xi \to \RN$ with $\int \| \varphi \|^p d\mu <
\infty$, and let 
\begin{equation} 
\label{defS}
{\mathcal S}^p_F \eqdef 
\{ \varphi \in {\mathcal L}^p(\Xi, {\mcT}, \mu; \RN) \, : \, 
\varphi(\xi) \in F(\xi) \ \mu-\text{a.e.} \} \, .
\end{equation} 
If ${\mathcal S}^1_F \neq \emptyset$, then the function $F$ is said to be 
integrable.  The \emph{Aumann integral} \cite{aum-65,aub-fra-90} of $F$ is the 
set 
\[
\int F d\mu \eqdef \left\{ \int_\Xi \varphi d\mu \ : \ 
  \varphi \in {\mathcal S}^1_F \right\} \, .
\]

\subsection{Random Maximal Monotone Operators} 
\label{mon-mes}

Consider the function $A : \Xi \to {\mathcal M}$. Note that the graph
$\graph(A(\xi,\cdot))$ of any element $A(\xi,\cdot)$ is a closed subset of $\RN
\times \RN$ by the maximality of $A(\xi,\cdot)$ \cite[Prop.~2.5]{bre-livre73}.
Assume that the function $\xi \mapsto \graph(A(\xi,\cdot))$ is measurable as a
closed set-valued $\Xi\rightrightarrows \RN\times\RN$ function.  It is shown in
\cite[Ch.~2]{att-79} that this is equivalent to saying that the function $\xi
\mapsto (I + \gamma A(\xi,\cdot))^{-1} x$ is measurable from $\Xi$ to $\RN$ for
any $\gamma > 0$ and any $x \in \RN$.  If the domain of $A(\xi,\cdot)$ is
represented by $D(\xi)$, the measurability of $\xi \mapsto
\graph(A(\xi,\cdot))$ implies that the set-valued function $\xi \mapsto
\clos(D(\xi))$ is measurable.  Moreover, recalling that $A(\xi, x)$ is the
image of a given $x \in \RN$ under the operator $A(\xi,\cdot)$, the set-valued
function $\xi \mapsto A(\xi, x)$ is measurable \cite[Ch.~2]{att-79}. 
Given $x\in D(\xi)$, the element of least norm in 
$A(\xi, x)$ is denoted as $A_0(\xi, x)$. In other words, 
$A_0(\xi, x) = \proj_{A(\xi,x)}(0)$. 
It is known that the
function $\xi \mapsto A_0(\xi, x)$ is measurable \cite[Ch.~2]{att-79}. 

For any $\gamma > 0$, the resolvent of $A(\xi,\cdot)$ is represented by 
\[
J_{\gamma}(\xi, x) \eqdef ( I + \gamma A(\xi,\cdot) )^{-1}(x) . 
\]
As we know, $J_{\gamma}(\xi, \cdot)$ is 
a non-expansive function on $\RR^N$. Since $J_{\gamma}(\xi, x)$ is 
measurable in $\xi$ and continuous in $x$, Carath\'eodory's theorem 
shows that the function $J_\gamma : \Xi \times \RN \to \RN$ is ${\mcT} \otimes 
{\mcB(\RN)}$ measurable. We also introduce the Yosida approximation 
$A_\gamma(\xi,\cdot)$ of $A(\xi,\cdot)$, which is defined for any $\gamma > 0$ 
as the ${\mcT} \otimes {\mcB}(\RN)$ measurable function 
\[
A_\gamma(\xi, x) \eqdef \frac{x - J_\gamma(\xi, x)}{\gamma} \, . 
\]
The function $A_\gamma(\xi, \cdot)$ is a $\gamma^{-1}$-Lipschitz continuous
function that 
satisfies $\| A_\gamma(\xi, x) \| \uparrow \| A_0(\xi, x) \|$ and 
$A_\gamma(\xi, x) \to A_0(\xi, x)$ for any $x \in D(\xi)$ when 
$\gamma \downarrow 0$. Moreover, the inclusion 
$A_\gamma(\xi, x) \in A(\xi, J_\gamma(\xi, x))$ holds true for all
$x \in \RR^N$ \cite{bre-livre73,bau-com-livre11}. 

The \emph{essential intersection} $\mD$ 
of the domains $D(\xi)$ is \cite{hu-these-77} 
\[
\mD \eqdef \bigcup_{E \in {\mcT} : \mu(E) = 0} \ 
\bigcap_{\xi \in \Xi \setminus E} D(\xi) \, , 
\]
in other words, 
$x \in \mD  \ \Leftrightarrow \ \mu(\{ \xi \, : \, x \in D(\xi) \}) = 1$. 
Let us assume that $\mD \neq \emptyset$ and that this function is integrable
for each $x \in \mD$. On $\mD$, we define $\mA$ as the Aumann integral 
\[
\mA(x) \eqdef \int_\Xi A(\xi, x) \mu(d\xi) \, . 
\]
One can immediately see that the operator $\mA : \mD \rightrightarrows \RN$ 
so defined is a monotone operator.

\subsection{Evolution Equations and Almost Sure APT} 
\label{evolution}
Given that $\sA \in \mathcal M$, consider the differential inclusion 
\begin{equation}
\label{di-gal} 
 \dot z(t) \in -\sA(z(t)) \quad \text{a.e. on } \RR_+, 
\quad z(0) = z_0 , 
\end{equation} 
for a given $z_0$ in $\dom(\sA)$. It is known from 
\cite{bre-livre73,aub-cel-(livre)84} that for any $z_0\in \dom(\sA)$, 
there exists a unique absolutely continuous function $z:\RR_+ \to \RN$ 
satisfying~\eqref{di-gal} - referred to as the \emph{solution} 
to~\eqref{di-gal}. Consider the map 
\[
\Psi : \dom(\sA) \times \RR_+ \to \dom(\sA), \quad (z_0, t) \mapsto z(t), 
\]
where 
$z(t)$ is the solution to~\eqref{di-gal} with the initial value $z_0$. Then,  
for any $t\geq 0$, $\Psi(\cdot, t)$ is a non-expansive map from $\dom(\sA)$ to 
$\dom(\sA)$ who can be extended by continuity to a non-expansive map from
$\clos(\dom(\sA))$ to $\clos(\dom(\sA))$ that we still denote as 
$\Psi(\cdot,t)$ \cite{bre-livre73,aub-cel-(livre)84}. The function $\Psi$ so 
defined is a \emph{semiflow} on the set $\clos(\dom(\sA)) \times \RR_+$, being 
a continuous function from $\clos(\dom(\sA)) \times \RR_+$ to 
$\clos(\dom(\sA))$, satisfying $\Psi(\cdot, 0) = I$ and 
$\Psi({z_0}, t+s) = \Psi(\Psi({z_0}, s), t)$ for every 
$z_0\in \clos(\dom(\sA))$, $t,s\geq 0$. The set 
$\gamma(x) \eqdef \{ \Psi(x, t) \, : \, t \geq 0 \}$ is the \emph{orbit} of 
$x$. Although orbits 
of $\Psi$ are not necessarily convergent in general, any solution 
to~(\ref{di-gal}) converges to a zero of $\sA$ (which is assumed to exist) 
whenever $\sA$ is \emph{demipositive} \cite{bru-75}. By demipositive, 
we mean that there exists $w\in Z(\sA)$ such that for
every sequence $(( u_n, v_n) \in \sA)$ such that $(u_n)$ converges to $u$ and
$\{ v_n \}$ is bounded, 
\[
\langle u_n - w, v_n \rangle \xrightarrow[n\to\infty]{}  0 
\quad \Rightarrow \quad u \in Z(\sA) \, .
\]

We now need to introduce some important notions associated with the semiflow $\Psi$.
A comprehensive treatment of the subject can be found
in~\cite{ben-hir-96,ben-(cours)99}. A set $S \subset \clos(\dom(\sA))$ is 
said to be \emph{invariant} for the semiflow $\Psi$ if $\Psi(S, t) = S$ for all 
$t\geq 0$. Given that $\varepsilon > 0$ and $T > 0$, a 
$(\varepsilon, T)$\emph{-pseudo orbit} from a point $a$ to a point $b$ in 
$\RR^N$ is a $n$-uple of partial orbits 
$(\{ \Psi(y_i, s) \, : \, s \in [0, t_i] \})_{i=0,\ldots, n-1}$ such that 
$t_i \geq T$ for $i=0,\ldots, n-1$, and 
\begin{align*}
\| y_0 - a \| &< \varepsilon ,  \\ 
\| \Psi(y_i, t_i) - y_{i+1} \| &< \varepsilon, \quad i=0, \ldots, n-1 , \\
y_n &= b .
\end{align*} 
Let $S$ be a compact and invariant set $S$ for $\Psi$. If for every 
$\varepsilon > 0$, $T > 0$ and every $a, b \in S$, there is an
$(\varepsilon, T)$-pseudo orbit from $a$ to $b$, then the set $S$ is said to be
\emph{Internally Chain Transitive} (ICT). 
We shall say that a random process $v(t)$ on $\RR_+$, who is valued in $\RR^N$,
is an almost sure 
asymptotic pseudo trajectory \cite{ben-hir-96,ben-hof-sor-05} for the 
differential inclusion~\eqref{di-gal} if 
\[
\sup_{s \in [0, T]} \| v(t+s) - \Psi(\proj_{\clos(\dom(\sA))}(v(t)), s) \|
\xrightarrow[t\to\infty]{} 0 \quad \text{a.s.} 
\] 
for any $T > 0$. We note that in the APT definition provided 
in~\cite{ben-hir-96,ben-hof-sor-05}, no projection is considered because the 
flow is defined in these references on the whole space. Projecting on 
$\clos(\dom(\sA))$ here does not alter the conclusions. Let 
$L(v) \eqdef \bigcap_{t\geq 0} \clos(v([t, \infty[))$  
be the \emph{limit set} of the trajectory $v(t)$, \emph{i.e.}, the set of 
the limits of the convergent subsequences $v(t_k)$ as $t_k \to\infty$. 
It is important to note that if $\{ v(t) \}_{t\in\RR_+}$ is bounded a.s., 
and if $v$ is an almost sure APT for~\eqref{di-gal}, then with probability
one, the compact set $L(v)$ is ICT for the semiflow $\Psi$ \cite{ben-hir-96}. 

The authors of \cite{ben-sch-00} establish a useful property of asymptotic 
pseudo trajectories pertaining to the asymptotic behavior of their 
empirical measures. We now consider that $v : \Omega \times \RR_+ \to \RR^N$ is
a random process on the probability space $(\Omega, \mcF, \PP)$ equipped with 
a filtration $( \mcF_t )_{t\in\RR_+}$. As we know, $v$ is said to be 
\emph{progressively measurable} if for each $t \geq 0$, the restriction to 
$\Omega \times [0, t]$ of $v$ is $\mcF_t \otimes \mcB([0, t])$-measurable, 
where $\mcB([0, t])$ is the Borel field over $[0, t]$. 
For $t \geq 0$, the \emph{empirical measure} $\nu_t(\omega, \cdot)$ of $v$ is 
then the random probability measure, defined by the identity 
\[ 
\int f(x) \, \nu_t(\omega, dx) = 
   \frac 1t \int_0^t f(v(\omega, s)) \, ds , 
\] 
for any measurable function $f : \RR^N \to \RR_+$. We also note that a
probability measure $\nu$ on $\RR^N$ is said to be \emph{invariant} for the
semiflow $\Psi$ if 
\[
\int f(x) \, \nu(dx) = \int f(\Phi(x,t)) \, \nu(dx) 
\]
for any $t \geq 0$ and any measurable function $f : \RR^N \to \RR_+$. 

Now, if $v$ is progressively measurable and if it is an almost sure APT for the
semiflow $\Psi$, then on a probability one set, all of the accumulation points
of the set $\{ \nu_t(\omega, \cdot) \}_{t\geq 0}$ for the weak convergence of
probability measures are invariant measures for $\Psi$
\cite[Th.~1]{ben-sch-00}. \footnote{The result is stated in \cite{ben-sch-00}
when $v$ is a so-called \emph{weak APT}. It turns out that any almost sure APT
is a weak APT by L\'evy's conditional form of Borel-Cantelli's lemma. }

\section{Results}
\label{sec:results} 
\subsection{Algorithm Description and Main Results} 
\label{main} 

Let $B:\Xi\to {\mathcal M}$ be a mapping such that, similarly to the mapping
$A$ introduced in Section~\ref{mon-mes}, the function $\xi \mapsto
\graph(B(\xi,\cdot))$ is measurable. Moreover, we assume throughout the paper
that $\dom(B(\xi,\cdot)) = \RN$ for almost every $\xi\in \Xi$. We also assume
that for every $x \in \RN$, $B(\cdot, x)$ is integrable, and we set $\mB(x)
\eqdef \int B(\xi,x)\mu(d\xi)$. Note that $\dom \mB = \RN$.  Let $(u_n)_{n\in
\NN^*}$ be an iid sequence of random variables from a probability space
$(\Omega, {\mcF}, \PP)$ to $(\Xi, \mcT)$ having the distribution $\mu$.
Starting with some arbitrary $x_0 \in \RN$, our purpose is to study the
behavior of the iterates 
\begin{align}
  x_{n+1} &= J_{\gamma_{n+1}}(u_{n+1}, x_n-\gamma_{n+1}b(u_{n+1},x_n)), \qquad (n \in \NN),
\label{fbr}
\end{align}
where the positive sequence $(\gamma_n)_{n\in\NN^*}$ belongs to 
$\ell^2\setminus\ell^1$, and where $b$ is a measurable map on 
$(\Xi\times \RN, {\mcT}\otimes {\mcB}(\RN))\to (\RN, {\mcB}(\RN))$
such that for every 
$x\in \RN$, $b(\,.\,,x) \in {\mathcal S}^1_{B(\,.\,,x)}$ \eqref{defS}.    
A possible choice for $b$ is $b(\xi,x)=B_0(\xi,x)$, which is 
${\mcT}\otimes {\mcB}(\RN)$--measurable, as the limit as $\gamma\downarrow 0$ 
of $B_\gamma(\xi, x)$. 
We define the affine interpolated process as 
\begin{equation}
\label{defx} 
x(t) \eqdef x_n + \frac{x_{n+1} - x_n}{\gamma_{n+1}} (t-\tau_n) 
\end{equation} 
for every $t\in[\tau_n, \tau_{n+1}[$, where $\tau_n=\sum_{1}^n \gamma_k$. 
Consider the differential inclusion
\begin{equation}
\label{di} 
\left\{
  \begin{array}[h]{l}
 \dot z(t) \in -(\mA+\mB)(z(t)), \quad \forall t \in \RR_+ \ \text{a.e.}, \\
z(0) = z_0\,.
\end{array}\right.
\end{equation}
If $\mA+\mB$ is maximal, then for any $z_0\in \mD$, (\ref{di}) has a unique 
solution, in which case, $\Phi : \clos(\mD) \times \RR_+ \to \clos(\mD)$ will
represent the semiflow associated to~\eqref{di}. 

Before stating our main result, we need to make a preliminary remark.  A point
$x_\star$ is an element of $\mZ = Z(\mA+\mB)$ if and only if there exists
$\varphi \in {\mathcal S}^1_{A(\cdot, x_\star)}$ and $\psi \in {\mathcal
S}^1_{B(\cdot, x_\star)}$ such that  $\int \varphi d\mu + \int \psi d\mu=0$. We
will refer to a couple $(\varphi,\psi)$ of this type as a \emph{representation}
of the zero $x_\star$. Moreover, in Theorem~\ref{th-apt} below, we shall assume
that there exists such a zero $x_\star$ for which the above functions
$\varphi$ and $\psi$ can be chosen in ${\mathcal L}^{2p}(\Xi, {\mcT}, \mu;
\RN)$, where $p\geq 1$ is some integer possibly strictly larger than one. We
thus introduce the set of $2p$-integrable representations
$$
{\mathcal R}_{2p}(x_\star) = \left\{(\varphi,\psi)\in  {\mathcal S}^{2p}_{A(\cdot, x_\star)}\times {\mathcal S}^{2p}_{B(\cdot, x_\star)}\,:\, \int \varphi d\mu + \int \psi d\mu=0\right\}\,.
$$
We let $\Pi(\xi, .)$ be the projection operator onto  
$\clos(D(\xi))$, and $d(\xi, \cdot)$ (respectively $\bs d(\cdot)$) be the 
distance function to $D(\xi)$ (respectively to $\mD$). 

\begin{theorem}
\label{th-apt} 
Assume the following facts: 
\begin{enumerate}

\item\label{Amax} 
The monotone operator $\mA$ is maximal. 

\item\label{mom}
There exists an integer $p \geq 1$ and a point $x_\star \in \mZ$ such that 
${\mathcal R}_{2p}(x_\star)\neq\emptyset$.

\item\label{cpct}
For any compact set $K$ of $\RN$, there exists $\varepsilon \in ]0,1]$ such 
that
\[
\sup_{x\in K \cap \mD} \int \| A_0(\xi, x) \|^{1+\varepsilon} 
   \, \mu(d\xi)  < \infty. 
\]
Moreover, there exists $y_0 \in \mD$ such that 
\[
\int \| A_0(\xi, y_0) \|^{1+1/\varepsilon} 
   \, \mu(d\xi)  < \infty \, . 
\] 
\item\label{dist}
There exists $C > 0$ such that for all $x \in \RR^N$, 
\[
\int d(\xi, x)^2 \mu(d\xi) \geq C \bs d(x)^2 \, , 
\]
and furthermore, $\gamma_{n+1} / \gamma_n \to 1$. 

\item\label{PI-J}
There exists $C > 0$ such that for any $x \in \RR^N$ and any 
$\gamma > 0$, 
\begin{align*}
  &\frac 1{\gamma^4}\int \| J_\gamma(\xi, x) - \Pi(\xi,x) \|^4 \mu(d\xi) \leq  C ( 1 + \| x \|^{2p} ) \, , 
\end{align*}
where the integer $p$ is specified in~\ref{mom}. 

\item\label{Bquad}
There exists $M:\Xi\to\RR_+$ such that $M^{2p}$ is $\mu$-integrable, and for 
all $x \in \RR^N$, 
$\|b(\xi,x)\|  \leq M(\xi) (1+\|x\|)$. 
Moreover, there exists a constant $C > 0$ such that 
$\int \|b(\xi,x)\|^4  \mu(d\xi) \leq C ( 1 + \| x \|^{2p} )$.

\end{enumerate}
Then, the monotone operator $\mA+\mB$ is maximal. Moreover, with probability 
one, the continuous time process $x(t)$ defined by \eqref{defx} is 
bounded and is an APT of the differential inclusion~\eqref{di}.  
\end{theorem} 

Let us now discuss our assumptions.  Sufficient conditions for the maximality of
$\mA$ are provided below in Sections~\ref{sec:maximality} and~\ref{A=dG}. 
Assumption~\ref{mom} is relatively weak and easy to check. If we set 
$\varepsilon=1$, then Assumption~\ref{cpct} can be replaced with the stronger
condition stating that for any compact set $K$ of $\RN$, 
\[
\sup_{x\in K \cap \mD} \int \| A_0(\xi, x) \|^{2} 
   \, \mu(d\xi)  < \infty \, . 
\]
For more insight on the above assumption, let us compare it with
the standard Robbins-Monro algorithm $y_{n+1} = y_n +\gamma_{n+1}
H(y_n,\xi_{n+1})$, where $H$ is some measurable function.  In order to ensure
the almost-sure boundedness of $(y_n)$, it is standard to assume that 
$\|H(y,\xi)\|\leq M(\xi)(1+\|y\|)$ for every $(y,\xi)$ and for some
square-integrable r.v.~$M(\xi)$ \cite{delyon:2000}. As far as our algorithm is
concerned, a similar assumption is needed on the operator $B$, but on the other
hand, no such assumption is needed on the operator $A$.  
Assumption~\ref{cpct} is weaker.  Otherwise stated, when a random operator is
used through its resolvent, there is no need to require the
``linear growth'' condition often assumed in the stochastic approximation
literature. 

Assumption~\ref{dist} is quite weak, and is easy to illustrate in the case
where $\mu$ is a finite sum of Dirac measures.  Following
\cite{bauschke1999strong}, we say that a finite collection of closed and convex
subsets $(\mathcal C_1,\dots,\mathcal C_m)$ over some Euclidean space is 
\emph{linearly regular} if there exists $\kappa>0$ such that for every $x$, 
\[
\max_{i=1\dots m} d(x,\mathcal C_i)\geq \kappa d(x,\mathcal C) ,  \quad 
\text{where}\ \mathcal C=\bigcap_{i=1}^m \mathcal C_i\, ,
\]
and where implicitely $\mathcal C\neq \emptyset$. Sufficient conditions for a 
collection of sets to satisfy the above condition can be found in 
\cite{bauschke1999strong} and the references therein. Note that this condition 
implies the
% standard condition qualification $\cap_i\mathrm{ri}(C_i)\neq\emptyset$ where
% $\mathrm{ri}$ stands for the relative interior.
so-called strong conical hull intersection property 
$N_{\mathcal C}(x) = \sum_{i=1}^mN_{\mathcal C_i}(x)$ for every 
$x\in \mathcal C$, where $N_{\mathcal C}(x)$ is, as we recall, the normal cone 
to ${\mathcal C}$ at the point $x$. 

Let us finally discuss Assumption~\ref{PI-J}. As $\gamma\to 0$, it is known
that $J_\gamma(\xi,x)$ converges to $\Pi(\xi,x)$ for every $(\xi,x)$.
Moreover, Assumption~\ref{PI-J} provides a control on the convergence rate. The
fourth moment of $\| J_\gamma(\xi, x) - \Pi(\xi,x) \|$ is assumed to vanish at
the rate $\gamma^4$ with a multiplicative factor of the order $\|x\|^{2p}$.  The
integer $p$ can potentially be as large as needed, provided that one is
able to find a zero $x_\star$ satisfying Assumption~\ref{mom}.  In the special
case where $A(\xi,\,.\,)$ coincides with the subdifferential of the convex
function $f(\xi,\,.\,)$, Assumption~\ref{PI-J} holds under the sufficient
condition that for almost every $\xi$ and for every $x\in D(\xi)$,
\begin{equation}
\|\partial_x f_0(\xi,x)\|\leq M'(\xi)(1+\|x\|^{p/2})\, , 
\label{eq:partialf0}
\end{equation}
where $\partial_x f_0(\xi,x)$ is the smallest norm element of the
subdifferential of $f(\xi,\,.\,)$ at point $x$, and where $M'(\xi)$ is a
positive r.v.~with a finite fourth moment. Indeed, in this case, the
resolvent $J_\gamma(\xi, x)$ coincides with $\prox_{\gamma f(\xi,\,.\,)}(x)$, 
and by \cite{bia-(arxiv)15},
\begin{align*}
  \frac 1\gamma\|J_\gamma(\xi,x)-\Pi(\xi,x)\|
&\leq  2\|\partial f_0(\xi,\Pi(\xi,x))\| \,. 
\end{align*}
As a consequence, Assumption~\ref{PI-J} stems from (\ref{eq:partialf0}) and the non-expansiveness of $\Pi(\xi,\,.\,)$.

The results of Theorem~\ref{th-apt} can first be used to study the convergence 
of the sequence $(\bar x_n)$ of empirical means, defined by
$$
\bar x_n \eqdef \frac{\sum_{k=1}^n \gamma_k x_k}{\sum_{k=1}^n \gamma_k} \, .
$$
\begin{corollary}
\label{avg} 
Let the assumptions in the statement of Theorem~\emph{\ref{th-apt}} hold true. 
Assume that for any $x_\star \in \mZ$, the set ${\mathcal R}_{2}(x_\star)$ is 
not empty. Then, for any initial value $x_0$, the sequence $(\bar x_n)$ of 
empirical means converges almost surely as $n\to\infty$ to a random variable
$U$, whose support lies in~$\mZ$. 
\end{corollary} 

Let us now consider the issue of the convergence of the sequence $(x_n)$ to a
point of $\mZ$.  Note that the conditions of Theorem~\ref{th-apt} are generally
insufficient to ensure that $x_n$ converges. A counterexample is obtained by
setting $N=2$ and taking $\mA$ as a $\pi/2$-rotation matrix, $\mB=0$
\cite[Sec.~6]{pey-sor-10}.  However, the statement will be proved valid
when $\mA+\mB$ is assumed demipositive.  We start by listing some
known verifiable conditions ensuring that the maximal monotone operator
$\mA+\mB$ is demipositive: 
\begin{enumerate}
\item\label{Gam0} $\mA+\mB = \partial G$, where $G \in \Gamma_0$ has a 
minimum.
\item\label{I-T} $\mA+\mB = I - T$, where $T$ is a non-expansive mapping having a fixed point. 
\item\label{int} The interior of $\mZ$ is not empty.  
\item\label{3m} 
$\mZ \neq \emptyset$ and $\mA+\mB$ is $3$-monotone, \emph{i.e.}, for every 
triple $(x_i, y_i) \in \mA+\mB$ for $i=1,2,3$, it holds that 
$\sum_{i=1}^3 \langle y_i, x_i - x_{i-1} \rangle \geq 0$ by setting 
$x_0 = x_3$. 
\item\label{smon}
$\mA+\mB$ is strongly monotone, \emph{i.e.}, 
$\langle x_1 - x_2, y_1 - y_2 \rangle \geq \alpha \| x_1 - x_2 \|^2$ for 
some $\alpha > 0$ and for all $(x_1, y_1)$ and $(x_2, y_2)$ in $\mA+\mB$. 

\item\label{coco}
$\mZ \neq \emptyset$ and $\mA+\mB$ is cocoercive, \emph{i.e.}, 
$\langle x_1 - x_2, y_1 - y_2 \rangle \geq \alpha \| y_1 - y_2 \|^2$ for 
some $\alpha > 0$ and for all $(x_1, y_1)$ and $(x_2, y_2)$ in $\mA+\mB$. 
\end{enumerate} 
The above conditions can be found in~\cite{pey-sor-10}. Specifically, 
conditions~\ref{Gam0}--\ref{int} can be found in \cite{bru-75}, while 
Condition~\ref{3m} can be found in~\cite{paz-78}. Conditions~\ref{smon} 
and~\ref{coco} can be easily verified to lead to the demipositivity of
$\mA + \mB$. 
Condition~\ref{Gam0} is further discussed in Section~\ref{A=dG} below. 
Condition~\ref{I-T} is satisfied if $\mZ \neq \emptyset$ and if for any
$\xi$, the operator $I - (A+B)(\xi,\cdot)$ is a non-expansive mapping.
Condition~\ref{3m} is satisfied if $\mZ \neq \emptyset$ and if all the 
operators $(A+B)(\xi,\cdot)$ are $3$-monotone. The last two conditions are 
most often easily verifiable. \\
We now have: 

\begin{corollary} 
\label{1/2p} 
Let the assumptions in the statement of Theorem~\emph{\ref{th-apt}} hold true. 
Assume in addition that the operator $\mA+\mB$ is demipositive, and that for 
any $x_\star \in \mZ$, the set ${\mathcal R}_{2}(x_\star)$ is not empty. 
Then, for any initial value $x_0$, there exists a random variable $U$,
supported by $\mZ$, such that $x_n \to U$ almost surely as $n\to\infty$. 
\end{corollary} 

We now address the important problem of the maximality of $\mA$. 

\subsection{Maximality of $\mA$} 
\label{sec:maximality}

By extending a well-known result on the maximality of the sum of two maximal
monotone operators, it is obvious that $\mA$ is maximal in the case where $\mu$
is a finite sum of Dirac measures and where the interior of $\mD$ is not empty
\cite{bre-livre73,bau-com-livre11}. For more general measures $\mu$, we have
the following result.

\begin{proposition} 
\label{max} 
Assume the following: 
\begin{enumerate} 
\item 
\label{intD}
The interior of $\mD$ is not empty, and there exists a closed ball in $\mD$ 
such that $\| A_0(\xi, x) \| \leq M(\xi)$ for any $x$ in this ball, 
and such that $M(\xi)$ is $\mu$-integrable. 
\item 
\label{1+eps}
For any compact set $K$ of $\RN$, there exists $\varepsilon > 0$ such that
\[
\sup_{x\in K \cap \mD} \int \| A_0(\xi, x) \|^{1+\varepsilon} 
   \, \mu(d\xi)  < \infty. 
\]
Moreover, there exists $y_0 \in \mD$ such that 
\[
\int \| A_0(\xi, y_0) \|^{1+1/\varepsilon} 
   \, \mu(d\xi)  < \infty \, . 
\]
\item 
\label{d>d}
There exists $C > 0$ such that for any $x \in \RN$, 
\[
\int d(\xi, x) \mu(d\xi) \geq C \bs d(x) . 
\]
\item 
\label{J-Pi} 
$\displaystyle{\int \| J_\gamma(\xi, x) - \Pi(\xi,x) \| \mu(d\xi) 
\leq \gamma C(x)}$, where $C(x)$ is bounded on compact sets of $\RN$. 
\end{enumerate} 
Then, the monotone operator $\mA$ is maximal. 
\end{proposition} 

\section{Application to Convex Optimization}
\label{sec-appli} 

We start this section by briefly reproducing some known results related to the
case where $A(\xi,\cdot)$ is the subdifferential of a proper, closed and convex 
function $g(\xi, \cdot)$. 

\subsection{Known Facts About the Aumann Integral of Subdifferentials} 
\label{A=dG} 

A function $g : \Xi \times \RR^N \to ]-\infty, \infty]$ is called
a \emph{normal integrand}~\cite{roc-69(mes)} if the set-valued mapping
$\xi \mapsto \epi g(\xi, \cdot)$ is closed-valued and measurable.
Let us assume in addition that $g(\xi, \cdot)$ is convex and proper 
for every $\xi$. 

Consider the case where $A(\xi,\cdot)=\partial g(\xi, \cdot)$. 
The mean operator $\mA$ is given by\footnote{By \cite{att-79,roc-wets-livre98},
the mapping $A : \Xi \to {\mathcal M}$,  
defined as $A(\xi,\cdot) = \partial g(\xi, \cdot)$, is measurable in the sense 
of Section~\ref{mon-mes}.}
$\mA(x) = \int \partial g(\xi,x) \mu(d\xi)$. 
Under some general conditions stated in~\cite{wal-wet-69}, the integral and 
the subdifferential can be exchanged in this expression. In this case, 
$\mA(x) = \partial G(x)$, where 
$G(x) = \int g(\xi, x) \, \mu(d\xi)$. This integral is defined as the 
sum 
\[
\int_{\{\xi \, : \, g(\xi, x) \in \RR_+ \}} 
           g(\xi, x) \, \mu(d\xi)  + 
\int_{\{\xi \, : \, g(\xi, x) \in ]-\infty, 0[\}} 
           g(\xi, x) \, \mu(d\xi)  +  I(x) \, , 
\]
where 
\[
I(x) = \left\{\begin{array}{cl} + \infty, &\text{if } 
         \mu(\{\xi : g(\xi, x) = \infty \}) > 0, \\
0, &\text{otherwise} \, , \end{array}\right. 
\]
and where the convention $(+\infty) + (-\infty) = + \infty$ is used. The 
function $G$ is a lower semi continuous and convex function if $G(x) > -\infty$ 
for all $x$~\cite{wal-wet-69}. Assuming in addition that $G$ is proper, the
identity $\mA = \partial G$ ensures that the operator $\mA$ is monotone, 
maximal, and demipositive, and that the zeros of $\mA$ are the minimizers of 
$G$.

\subsection{A Constrained Optimization Problem}

Let $(\sX,\mcX, \nu)$ be a probability space. Let
the functions $f : \sX \times \RR^N \to ]-\infty, \infty[$ and $g : \sX \times
\RR^N \to ]-\infty, \infty[$ be normal convex integrands. Here we assume that
$g$ is finite everywhere to simplify the presentation. However we note that the
results can be extended to the case where $g$ is allowed to take the value
$+\infty$. Recall the optimization problem
\begin{equation}
\min_{x\in \mathcal C} F(x) + G(x),  \quad \mathcal C = \bigcap_{i=1}^m \mathcal C_i,
\label{eq:constrained}
\end{equation}
where $F(x) = \int f(\eta, x) \nu(d\eta)$,  $G(x) = \int g(\eta, x) \nu(d\eta)$
and  $\mathcal C_1,\ldots, \mathcal C_m$ are closed and convex sets. 
Consider a measurable function $\tilde\nabla f:\sX\times \RN\to \RR$ such that 
for every $\eta\in \sX$ and $x\in \RN$, $\tilde\nabla f(\eta,x)$ is a subgradient of $f(\eta,\,.\,)$ 
at $x$. Let $(v_n)_n$ be an iid sequence on $\sX$ with probability distribution $\nu$.
Finally, let $(I_n)$ be an iid sequence on $\{0,1,\dots,m\}$ with distribution
$\alpha_i = \PP(I_1=i) > 0$ for every $i$.
We consider the iterates
\begin{equation}
x_{n+1} = \left\{\begin{array}{ll}
\prox_{\alpha_0^{-1} \gamma_{n+1} g(v_{n+1}, \cdot)} ( x_n - \gamma_{n+1} \tilde\nabla f(v_{n+1},x_n) ) , 
  &\quad \text{if} \ I_{n+1} = 0, \\
\proj_{\mathcal C_{I_{n+1}}}(x_n - \gamma_{n+1} \tilde\nabla f(v_{n+1},x_n) ), 
 & \quad \text{otherwise}.
\end{array}\right.
\label{eq:algoOpt}  
\end{equation}
We recall that $\partial g_0(\eta,x)$ is the least norm element of the 
subdifferential of $g(\eta,\,.\,)$ at $x$. Given $H \subset \RN$, we use the 
notation $|H|=\sup \{\|v\|:v\in H\}$.
\begin{corollary}
\label{coro:optim}
  We assume the following. Let $p\geq 1$ be an integer.
  \begin{enumerate}
\item For every $x\in \RR^N$, 
$\int |f(\eta,x)|\nu(d\eta)+\int |g(\eta,x)|\nu(d\eta)<\infty $. 
  \item For any solution $x_\star$ to Problem~\emph{\eqref{eq:constrained}}, 
there exists a measurable function $M_\star:\sX\to \RR_+$
such that $\int M_\star(\eta)^2\nu(d\eta)<\infty$, and for all $\eta\in \sX$, 
\[
|\partial f(\eta,x_\star)|+|\partial g(\eta,x_\star)|\leq M_\star(\eta)\, .
\] 
Moreover, there exists a solution $x_\star$ for which 
$\int M_\star(\eta)^{2p}\nu(d\eta)<\infty$.
\item For any compact set $K$ of $\RN$, there exists $\varepsilon \in ]0,1]$ 
such that
\[
\sup_{x\in K} \EE \| \partial g_0(\Theta, x) \|^{1+\varepsilon} < \infty \,.
\]
Moreover, there exists $y_0 \in \mathcal C$ such that 
$\EE \| \partial g_0(\Theta, y_0) \|^{1+1/\varepsilon}  < \infty$.
 \item The closed and convex sets $\mathcal C_1,\dots,\mathcal C_m$ are 
linearly regular, \emph{i.e.},
    $$
    \exists \kappa>0, \forall x\in \RN,\, \max_{i=1,\dots,m} 
      \dist(x,\mathcal C_i)\geq \kappa\, \dist(x, \mathcal C) \, , 
    $$
where $\dist(x, S)$ denotes the distance of the point $x$ to the set $S$. 
Moreover, $\gamma_n/\gamma_{n+1}\to 1$. 

 \item There exists $M:\sX\to\RR$ such that 
$\int M(\eta)^{2p}\nu(d\eta)<\infty$, and 
\[
\forall (\eta,x)\in \sX\times \RN, \ 
\|\tilde \nabla f(\eta,x)\|\leq M(\eta)(1+\|x\|) \, .
\]

\item There exists $c>0$ such that $\forall x\in \RN$,
$\int\|\tilde \nabla f(\eta,x)\|^4\nu(d\eta)\leq c(1+\|x\|^{2p})$. 

  \end{enumerate}
Then, the sequence $(x_n)$ given by~\emph{\eqref{eq:algoOpt}} converges 
almost surely to a solution to Problem~\emph{\eqref{eq:constrained}}.
\end{corollary}

\section{Related Works}
\label{sec-stArt}

The problem of minimizing an objective function in a noisy environment has
brought forth a very rich body of literature in the field of stochastic 
approximation \cite{ben-(cours)99,kus-yin-(livre)03}. In the framework of this 
paper, most of this literature examines the evolution of the projected 
stochastic gradient or subgradient algorithm, where the projection is 
made on a fixed constraining set. 

In the case where the constraining set has a complicated structure, an
incremental minimization algorithm with random constraint updates has been
proposed in \cite{ned-11}, where a deterministic convex function $f$ is
minimized on a finite intersection of closed and convex constraining sets. The
algorithm developed in \cite{ned-11} consists of a subgradient step over the
objective $f$ followed by an update step towards a randomly chosen constraining
set. Using the same principle, a distributed algorithm involving an additional
consensus step has been proposed in \cite{lee-ned-13}.  Random iterations
involving proximal and subgradient operators
were considered in \cite{ber-11} and in \cite{wan-ber-(techreport)13}. In
\cite{wan-ber-(techreport)13}, the functions $g(\xi,\,.\,)$ are supposed to
have a full domain, to satisfy $\|g(\xi,x)-g(\xi,y)\|\leq
L(\|x-y\|+1)$ for some constant $L$ which does not depend on $\xi$ and,
finally, are such that $\int \|g(\xi,x)\|^2\mu(d\xi)\leq L(1+\|x\|^2)$.  In the
present paper, such conditions are not needed.  

The algorithm~\eqref{fbr} can also be used to solve a variational 
inequality problem. Let $\mathcal C = \cap_{i=1}^m \mathcal C_i$ 
where ${\mathcal C}_1,\ldots, {\mathcal C}_m$ are closed and convex sets in $\RR^N$.
Consider the problem of finding 
$x_\star \in \mathcal C$ that solves the variational inequality 
\[
\forall x \in \mathcal C, \ \langle F(x_\star), x - x_\star \rangle \geq 0 \, , 
\]
where $F : \RR^N \to \RR^N$ is a monotone single-valued operator on 
$\RR^N$ \cite{wan-ber-15,kin-sta-(livre)00}. Since the projection on 
$\mathcal C$ is difficult, one can use the simple stochastic algorithm 
$x_{n+1} = \proj_{\mathcal C_{u_{n+1}}}(x_n - \gamma_{n+1} F(x_n) )$, where 
the random variables $u_n$ are distributed on the set $\{1,\ldots, m\}$. The 
variant where $F$ is itself an expectation can also be considered \emph{i.e.},
$F(x) = \int f(\xi, x) \mu(d\xi)$.
The work~\cite{wan-ber-15} addresses this context. 
In~\cite{wan-ber-15}, it is assumed
that $F$ is strongly monotone and that the  stochastic Lipschitz property
$\int \| f(\xi, x) - f(\xi, y) \|^2 \mu(d\xi) \leq C \| x - y \|^ 2$ holds,
where $C$ is a positive constant.  In our work, the strong monotonicity of $F$
is not needed, and the Lipschitz property is essentially replaced with the
condition $\| \tilde\nabla f(\xi, x) \| \leq M(\xi) (1 + \| x \|)$, where
$\tilde\nabla f(\xi, x)$ is a subgradient of $f(\xi,\cdot)$ at $x$ (for
instance, the least norm one), and $M(\xi)$ satisfies a moment condition.  

In the same vein as our paper, \cite{pas-79} considered a collection 
$\{ A(i, \cdot) \}_{i=1}^N$ of $N$ maximal monotone operators, and studied 
the iterations 
\[
y_{n+1} \in A(\sigma_{n+1}(1), x_n)\, , \ 
x_{n+1} = \prod_{i=2}^N (I + \gamma_{n+1} A(\sigma_{n+1}(i), \cdot))^{-1} 
(x_n - \gamma_{n+1} y_{n+1}) \, , 
\]
where $(\gamma_n) \in \ell^2\setminus\ell^1$, and where $(\sigma_n)$ is a
sequence of permutations of the set $\{1,\ldots, N\}$.  The convergence of
$(\bar x_n)$ to a zero of $\sum A(i,\cdot)$ is established in~\cite{pas-79}. 
In the recent paper~\cite{com-pes-(pafa)16}, a relaxed version of 
Algorithm~\eqref{fb-deter} is considered, where $\sB$ is cocoercive and where
its output, as well as the output of the resolvent of $\sA$, are subjected
to random errors. The convergence of the iterates to a zero of $\sA + \sB$ is 
established under summability assumptions on these errors. 

Regarding the convergence rate analysis, let us mention
\cite{atc-for-mou-14,rosasco2014convergence} which investigate the performance
of the algorithm 
$x_{n+1} = \prox_{\gamma_{n+1} g}( x_n - \gamma_{n+1} H_{n+1})$, where $H_{n+1}$
is a noisy estimate of the gradient $\nabla f(x_n)$.  The same algorithm is
addressed in \cite{rosasco2015stochastic}, where the proximity operator is
replaced by the resolvent of a fixed maximal monotone operator, and $H_{n+1}$
is replaced by a noisy version of a (single-valued) cocoercive operator
evaluated at $x_n$. The paper \cite{tou-tra-air-(arxiv)15} addresses the
statistical analysis of the empirical means of the estimates obtained from the
random proximal point algorithm. 

This paper follows the line of thought of the recent paper
\cite{bia-(arxiv)15}, who studies the behavior of the random iterates
$x_{n+1} = J_{n+1}(u_{n+1}, x_n)$ in a Hilbert space,
and establishes the convergence of the empirical means $\bar x_n$ towards a
zero of the mean operator $\mA(x) = \int A(\xi, x) \, \mu(d\xi)$.  In the
present paper, the proximal point algorithm is replaced with the more general
forward-backward algorithm. Thanks to the dynamic approach developed here,
the convergences of both $(\bar x_n)$ and $(x_n)$ are studied. 

Finally, it is worth noting that apart from the APT of Bena\"{\i}m and
Hirsch~\cite{ben-hir-96}, many authors have introduced alternative concepts to
analyze the asymptotic behavior of perturbed solutions to evolution systems.
An important one is the notion of \emph{almost-orbit} of
\cite{alv-pey-09,alvarez2010asymptotic}, and \cite{alv-pey-11}, which has 
been shown to be useful
to analyze certain perturbed solution to differential inclusions of the
form~(\ref{di-gal}).  The almost-orbit property is however more demanding than
the APT property, and is in general harder to verify, although it can lead to
finer convergence results. Fortunately, the concept of APT has been proven
sufficient here to guarantee that the interpolated process $x(t)$ almost surely
inherits both the ergodic and non-ergodic convergence properties of the orbits
of $\Phi$. 

\section{Proofs} 
\label{sec-prf} 

Let us start with the proof of Proposition~\ref{max} because it contains many 
elements of the proof of the main theorem. 

\subsection{Proof of Proposition \ref{max}} 

We recall that 
for any $\xi \in \Xi$ and any $\gamma > 0$, the Yosida approximation
$A_\gamma(\xi, \cdot)$ is a single-valued $\gamma^{-1}$-Lipschitz monotone 
operator defined on $\RN$. As a consequence, 
the operator $\mA^\gamma:\RN\to\RN$, given by
$\mA^{\gamma}(x) = \int A_{\gamma}(\xi, x) \mu(d\xi)$, 
is a single-valued, continuous, and monotone operator defined on $\RN$. 
As such, 
$\mA^\gamma$ is maximal \cite[Prop.~2.4]{bre-livre73}. Thus, given any 
$y \in \RN$, there exists $x^\gamma \in \RN$ such that 
$y = x^\gamma + \mA^\gamma(x^\gamma)$. We shall find a sequence 
$\gamma_n \to 0$ such that $x^{\gamma_n} \to x^\star \in \mD$ with 
$y - x^\star \in \mA x^\star$. The maximality of $\mA$ then follows by Minty's 
theorem~\cite{bre-livre73}.  \\ 
Let $z_0$ and $\rho$ be respectively the centre and the radius of the ball
referred to in Assumption~\ref{intD}, and set
\[
u(\xi) = z_0 + \rho 
\frac{A_\gamma(\xi, x^\gamma)}{\| A_\gamma(\xi, x^\gamma) \|} \in \mD \, , 
\]
where the convention $0/0=0$ is used. By the monotonicity of 
$A_\gamma(\xi,\cdot)$, 
\[
0 \leq \int \langle x^\gamma - u(\xi), 
 A_\gamma(\xi, x^\gamma) - A_\gamma(\xi,u(\xi)) \rangle \, 
             \mu(d\xi) .
\]
Writing $C = \int M(\xi) \mu(d\xi) < \infty$ 
(see Assumption~\ref{intD}), we obtain 
\begin{align*} 
\int \langle x^\gamma , A_\gamma(\xi, x^\gamma) \rangle \, \mu(d\xi) &= 
\langle x^\gamma, y \rangle - \| x^\gamma \|^2, \\ 
\int \langle - u(\xi), A_\gamma(\xi, x^\gamma) \rangle \, \mu(d\xi) &= 
\langle z_0, x^\gamma - y \rangle 
       - \rho \int \| A_\gamma(\xi, x^\gamma) \| \, \mu(d\xi), \\
\int | \langle x^\gamma, A_\gamma(\xi, u(\xi)) \rangle | \, 
\mu(d\xi) &\leq 
\| x^\gamma \| \int \| A_0(\xi, u(\xi) \| \, \mu(d\xi) \leq 
C \| x^\gamma \| , \\
\int | \langle u(\xi), A_\gamma(\xi, u(\xi)) \rangle | \, 
\mu(d\xi) &\leq C (\| z_0 \| + \rho) . 
\end{align*} 
Therefore, 
\[
\rho \int \| A_\gamma(\xi, x^\gamma) \| \, \mu(d\xi) 
+ \| x^\gamma \|^2 \leq 
\| x^\gamma \| ( \| y \| + \| z_0 \| + C) + C (\| z_0 \| + \rho) +\| z_0\|\,\|y\|\,. 
\] 
This shows that the sets $\{ \| x^\gamma \| \}$ and 
$\{ \int \| A_\gamma(\xi, x^\gamma) \| \, \mu(d\xi) \}$ are both
bounded. Writing $A_\gamma(\xi, x^\gamma) = \gamma^{-1} ( 
\Pi(\xi, x^\gamma) - J_\gamma(\xi, x^\gamma) )
+ \gamma^{-1} ( x^\gamma - \Pi(\xi, x^\gamma) )$, and using 
Assumption~\ref{J-Pi}, we obtain that the set 
$\{ \gamma^{-1} \int \| x^\gamma - \Pi(\xi, x^\gamma) \| \mu(d\xi) \}$ 
is bounded. By Assumption~\ref{d>d}, $\{ \bs d(x^\gamma) / \gamma \}$ is bounded. 
Given $x^\gamma$, let us choose $\tilde x^\gamma \in \mD$ such that 
$\| x^\gamma - \tilde x^\gamma \| \leq 2 \bs d(x^\gamma)$. By the boundedness of
$\{ \| x^\gamma \| \}$, there exists a compact set $K\subset \RN$ such that 
$\tilde x^\gamma \in K$. Associating a positive number $\varepsilon$ to $K$ as 
in Assumption~\ref{1+eps}, we obtain 
\begin{align*}
\int \| A_\gamma(\xi, & x^\gamma) \|^{1+\varepsilon} \, \mu(d\xi) \\
&\leq 2^\varepsilon \int \Bigl(
\|A_\gamma(\xi, \tilde x^\gamma )\|^{1+\varepsilon} 
+ \| A_\gamma(\xi, x^\gamma)  - 
             A_\gamma(\xi, \tilde x^\gamma )\|^{1+\varepsilon} 
\Bigr) \mu(d\xi) \\
&\leq 2^\varepsilon \int \| A_0(\xi, \tilde x^\gamma )\|^{1+\varepsilon} \,
\mu(d\xi) + 
2^{1+2\varepsilon} \Bigl|\frac{\bs d(x^\gamma)}{\gamma}\Bigr|^{1+\varepsilon} ,
\end{align*} 
which is bounded by a constant independent of $\gamma$ thanks to 
Assumption~\ref{1+eps}. Thus, the family of $\Xi \to \RN$ functions 
$\{ A_\gamma(\xi, x^\gamma) \}$ is bounded in the Banach space 
${\mathcal L}^{1+\varepsilon}(\Xi, {\mcT}, \mu; \RN)$. \\
Let us take a sequence $(\gamma_n, x^{\gamma_n})$ converging to $(0, x^\star)$.
Let us extract a subsequence (still denoted as $(n)$) from the sequence of 
indices $(n)$, in such a way that $(A_{\gamma_n}(\xi, x^{\gamma_n}))_{n}$ 
converges weakly in ${\mathcal L}^{1+\varepsilon}$ towards a function $f(\xi)$. 
By Mazur's theorem, there exists a function $J : \NN \to \NN$ and a sequence 
of sets of weights $( \{ \alpha_{k,n}, k=n\ldots, J(n) \, : \, 
\alpha_{k,n} \geq 0, \sum_{k=n}^{J(n)} \alpha_{k,n} = 1 \} )_{n}$ such 
that the sequence of functions
$(g_n(\xi)=\sum_{k=n}^{J(n)} \alpha_{k,n} 
                                  A_{\gamma_k}(\xi,x^{\gamma_k}))$ 
converges strongly to $f$ in ${\mathcal L}^{1+\varepsilon}$. Taking a further 
subsequence, we obtain the $\mu$-almost everywhere convergence of $(g_n)$ to $f$.
\\
Observe that $x^\star \in \clos(\mD)$ since $\bs d(x^{\gamma_n}) \to 0$.
Choose a sequence $(z_n)$ in $\mD$ that converges to $x^\star$, and for each 
$n$, let $T_n = \{ \xi \in \Xi \, : \, z_n \in D(\xi) \}$. Then, on the 
probability one set $T = \cap_n T_n$, it holds that 
$x^\star \in \clos(D(\xi))$. On the intersection of $T$ and the set
where $g_n \to f$, set
$\eta_n(\xi) = J_{\gamma_n}(\xi, x^{\gamma_n}) - x^\star$, and write
\[ 
\| \eta_n(\xi) \| \leq 
\| J_{\gamma_n}(\xi, x^{\gamma_n}) - J_{\gamma_n}(\xi, x^{\star}) \|  
 + \| J_{\gamma_n}(\xi, x^{\star}) - x^{\star} \| .  
\]
Since $J_{\gamma_n}(\xi, \cdot)$ is non-expansive and since 
$x^\star \in \clos(D(\xi))$, we have  
$\eta_n(\xi) \to_n 0$. Considering Assumption~\ref{1+eps}, we also have 
\begin{align*}
\| \eta_n(\xi) \| &\leq \| x^\star \| + 
\| J_{\gamma_n}(\xi, x^{\gamma_n}) - J_{\gamma_n}(\xi, y_0) \|  
 + \| J_{\gamma_n}(\xi, y_{0}) - y_{0} \|  + \| y_0 \|  \\ 
&\leq \| x^\star \| + \sup_\gamma \| x^\gamma \| + 2 \| y_0 \| + 
\| A_0(\xi, y_0) \| \, , 
\end{align*} 
when $\gamma_n \leq 1$. By Assumption \ref{1+eps} and the dominated 
convergence theorem, we obtain that $\eta_n \to 0$ in 
${\mathcal L}^{1+1/\varepsilon}$. With this in mind, 
\begin{multline*}
  \int | \langle \eta_n(\xi), A_{\gamma_n}(\xi,x^{\gamma_n}) \rangle |
  \mu(d\xi) \\ \leq \left(\int \| \eta_n(\xi)  \|^{1+1/\varepsilon}\mu(d\xi)\right)^{\varepsilon/(1+\varepsilon)} \left(\int \|
  A_{\gamma_n}(\xi,x^{\gamma_n}) \|^{1+\varepsilon}\mu(d\xi)\right)^{1/(1+\varepsilon)} \, , 
\end{multline*}
and the left-hand side converges to zero. Consequently, the random variable 
\[
e_n = 
\sum_{k=n}^{J(n)} \alpha_{k,n} 
\langle J_{\gamma_k}(\xi, x^{\gamma_k}) - x^\star, 
                      A_{\gamma_k}(\xi,x^{\gamma_k}) \rangle 
\]
converges to zero in probability, hence in the $\mu$-almost sure sense along a 
subsequence. Fix $\xi$ in this new probability one set, choose
arbitrarily a couple $(u,v) \in A(\xi,\cdot)$, and write 
\[
X_n = \sum_{k=n}^{J(n)} \langle u - J_{\gamma_k}(\xi, x^{\gamma_k}), 
\alpha_{k,n} v - \alpha_{k,n} A_{\gamma_k}(\xi,x^{\gamma_k}) \rangle . 
\]
It holds by the monotonicity of $A(\xi,\cdot)$ that $X_n \geq 0$. Writing 
\[
X_n = \langle u - x^\star, v - g_n(\xi) \rangle + e_n 
- \sum_{k=n}^{J(n)} \alpha_{k,n} \langle \eta_k, v \rangle  \, , 
\]
and making $n\to\infty$, we obtain that 
$\langle u - x^\star, v - f(\xi) \rangle \geq 0$. By the maximality of 
$A(\xi,\cdot)$, it holds that $(x^\star, f(\xi)) \in A(\xi,\cdot)$. \\
To conclude, we have 
\[
y = \sum_{k=n}^{J(n)} \alpha_{k,n} x^{\gamma_k} + 
\int g_n(\xi) \, \mu(d\xi) , 
\]
$\sum_{k=n}^{J(n)} \alpha_{k,n} x^{\gamma_k} \to_n x^\star \in \mD$, and 
$g_n \xrightarrow{{\mathcal L}^1(\mu)} 
     f \in {\mathcal S}^1_{A(\cdot, x^\star)}$. 
Making $n\to\infty$, we obtain 
$y - x^\star = \int f(\xi) \, \mu(d\xi) \in \mA(x^\star)$, which is the
desired result. 
\hfill\qed

\subsection{Proof of Theorem \ref{th-apt}}

Noting that $\dom \mB = \RR^N$ and using Assumption~\ref{Bquad} of 
Theorem~\ref{th-apt}, one can check that the assumptions of 
Proposition~\ref{max} are satisfied for $B$. The result is that $\mB$ is 
maximal. 
Because $\mB$ has a full domain and $\mA$ is maximal, $\mA+\mB$ is 
maximal by~\cite[Corollary 24.4]{bau-com-livre11}.
Thus, the first assertion of Theorem~\ref{th-apt} is shown, and moreover, the 
differential inclusion~\eqref{di} admits a unique solution, and the associated
semiflow $\Phi$ is well defined.

Defining $Y_\gamma(\xi,x) \eqdef A_\gamma(\xi,x-\gamma b(\xi,x))$, the iterates 
$x_n$ can be rewritten as
\begin{align*}
  x_{n+1} &= x_n - \gamma_{n+1} b(u_{n+1},x_n)-\gamma_{n+1} Y_{\gamma_{n+1}}(u_{n+1}, x_n)\\ 
  &= x_n
  - \gamma_{n+1} h_{\gamma_{n+1}}(x_n) + \gamma_{n+1} \eta_{n+1},
\end{align*}
where we define
\begin{align*}
  h_\gamma(x) &\eqdef \int (Y_\gamma(\xi,x) +b(\xi,x))\mu(d\xi)\, , 
\end{align*}
and 
$$
  \eta_{n+1} \eqdef -Y_{\gamma_{n+1}}(u_{n+1}, x_n) 
 + \EE_n Y_{\gamma_{n+1}}(u_{n+1}, x_n) -b(u_{n+1},x_n) 
 + \EE_n b(u_{n+1},x_n) \, , 
$$
where $\EE_n$ denotes the expectation conditionally to the sub $\sigma$-field 
$\sigma( u_1, \ldots, u_n)$ of $\mcF$ (we also write $\EE_0 = \EE$). 
Consider the martingale 
\[
M_n \eqdef \sum_{k=1}^n \gamma_{k} \eta_{k} \, , 
\]
and let $M(t)$ be the affine interpolated process, defined for any $n \in \NN$ 
and any $t \in [ \tau_n, \tau_{n+1}[$ as
\[
M(t) \eqdef M_n + \eta_{n+1} (t - \tau_n) = 
M_n + \frac{M_{n+1} - M_n}{\gamma_{n+1}} (t-\tau_n) . 
\]
For any $t \geq 0$, let
\[
r(t) \eqdef \max \{ k \geq 0 \ : \ \tau_k \leq t \} .
\]
Then, for any $t\geq 0$, we obtain
\begin{align} 
x(\tau_n + t) - x(\tau_n) &= 
- \int_{0}^{t} h_{\gamma_{r(\tau_n+s)+1}}(x_{r(\tau_n+s)}) \, ds 
+ M(\tau_n+t) - M(\tau_n) \nonumber \\ 
&= H(\tau_n+t) - H(\tau_n) + M(\tau_n+t) - M(\tau_n) \, , 
\label{eq-int} 
\end{align}  
where $H(t) \eqdef \int_0^t h_{\gamma_{r(s)+1}}(x_{r(s)}) \, ds$.  
The idea of the proof is to establish that on a $\PP$-probability one set, the 
sequence $( x(\tau_n + \cdot) )_{n\in\NN}$ of continuous time processes 
is equicontinuous and bounded. 
The accumulation points for the uniform convergence on a compact interval 
$[0, T]$ (who are guaranteed to exist by the Arzel\`a-Ascoli theorem) will be
shown to have the form 
\begin{equation}
z(t) - z(0) = - \lim_{n\to\infty} 
\int_{0}^{t} ds \int_{\Xi} \mu(d\xi) \, 
(Y_{\gamma_{r(\tau_n+s)+1}}(\xi, x_{r(\tau_n+s)}) + b(\xi,x_{r(\tau_n+s)}) )\, ,
\label{eq:zlim}
\end{equation}
where the limit is taken over a subsequence. We then show that the sequence of 
$\Xi \times [0,T] \to \RR^{2N}$ functions 
$((\xi,s)\mapsto Y_{\gamma_{r(\tau_n+s)+1}}(\xi, x_{r(\tau_n+s)}), b(\xi,x_{r(\tau_n+s)}))_n$ 
is bounded in the Banach space 
${\mathcal L}^{1+\varepsilon}(\Xi \times [0,T], \mu \otimes \lambda )$,  
where $\lambda$ is the Lebesgue measure on $[0,T]$. Analyzing the accumulation points
and following an approach similar to the one used in the proof of Proposition~\ref{max}, 
we prove that the limit in the right-hand side of~(\ref{eq:zlim}) coincides with
$$
z(t) - z(0) = - \lim_{n\to\infty} \int_{0}^{t} ds \left(\int_{\Xi} f^{(a)}(\xi,s) \mu(d\xi) +\int_{\Xi} f^{(b)}(\xi,s) \mu(d\xi) \right)\, , 
$$
where for almost every $s\in [0,T]$, $f^{(a)}(\cdot,s)$ and $f^{(b)}(\cdot,s)$
are integrable selections of $A(\cdot,s)$ and $B(\cdot,s)$, respectively. This
shows that $z$ satisfies the differential inclusion~\eqref{di}.  Hence, almost
surely, the accumulation points of the sequence of processes 
$( x(\tau_n + \cdot) )_{n\in\NN}$ are solutions to~\eqref{di}. Recalling that
the latter defines a semiflow 
$\Phi:\clos(\mD) \times \RR_+ \to \clos(\mD)$, it follows that the 
process $x(t)$ is a.s. an APT of~(\ref{di}).\\

Throughout the proof, $C$ refers to a positive constant, that can change from
line to line, but that remains independent of $n$. We use $c$, $c_1$, etc.~to
denote random variables on~$\Omega \to \RR_+$ that do not depend on $n$. For a
fixed event $\omega \in \Omega$, these will act as constants.  

\begin{proposition}
\label{opial1} 
Let Assumptions \emph{\ref{mom}} and~\emph{\ref{Bquad}} of 
Theorem~\emph{\ref{th-apt}} hold true. Then, 
\begin{enumerate}
\item The sequence $(x_n)$ is bounded almost surely and in 
${\mathcal L}^2(\Omega, \mcF, \PP; \RR^N)$. 
\item $\EE[ \sum_n \gamma_n^2 \int \| Y_{\gamma_n}(\xi,x_n) \|^2 \mu(d\xi) ] 
      < \infty$. 
\item\label{randfejer}
The sequence $(\| x_n - x_\star \|)_n$ converges almost surely. 
\end{enumerate} 
\end{proposition} 
\begin{proof}
Writing 
$\|x_{n+1}-x_\star\|^2 = 
\|x_{n}-x_\star\|^2+2\ps{x_{n+1}-x_n,x_n-x_\star} + \|x_{n+1}-x_n\|^2 
$, 
we obtain
\begin{multline*}
\|x_{n+1}-x_\star\|^2    
= \|x_{n}-x_\star\|^2 
 - 2\gamma_{n+1}\ps{Y_{\gamma_{n+1}}(u_{n+1},x_n),x_n-x_\star} \\
 - 2\gamma_{n+1}\ps{b(u_{n+1},x_n),x_n-x_\star} 
  + \gamma_{n+1}^2\|b(u_{n+1},x_n)+Y_{\gamma_{n+1}}(u_{n+1}, x_n)\|^2 . 
\end{multline*}
Thanks to Assumption~\ref{mom}, we can choose 
$\varphi\in {\mathcal S}_{A(\cdot,x_\star)}^2$ and 
$\psi \in {\mathcal S}_{B(\cdot,x_\star)}^1$ such that 
$0= \int (\varphi + \psi) d\mu$. 
Writing $u = u_{n+1}$, $\gamma = \gamma_{n+1}$, $Y_\gamma = Y_{\gamma_{n+1}}(u_{n+1},x_n)$, 
$J_\gamma = J_{\gamma_{n+1}}(u_{n+1},x_n - \gamma_{n+1} b(u_{n+1}, x_n))$, 
and $b = b(u_{n+1},x_n)$ for conciseness, and recalling that 
$Y_\gamma = (x - \gamma b - J_\gamma) / \gamma$, we write 
\begin{align*} 
\ps{Y_\gamma, x_n - x_\star} &= \ps{Y_\gamma - \varphi(u), J_\gamma - x_\star}
+ \gamma \ps{Y_\gamma - \varphi(u), Y_\gamma} + 
\gamma \ps{Y_\gamma - \varphi(u),b} \\
&\phantom{=} + \ps{\varphi(u), x_n - x_\star}  \\
&\geq 
\gamma \| Y_\gamma \|^2 - \gamma \ps{\varphi(u), Y_\gamma} 
+ \gamma \ps{Y_\gamma - \varphi(u),b} + \ps{\varphi(u), x_n - x_\star} , 
\end{align*}
since $Y_\gamma \in A(u, J_\gamma)$ and $A(\xi,\cdot)$ is monotone. By the 
monotonicity of $B(\xi,\cdot)$, we also have 
$\ps{b, x_n - x_\star} \geq \ps{\psi(u), x_n - x_\star}$. 
By expanding $\gamma^2 \| b + Y_\gamma \|^2$, we obtain altogether 
\begin{align}
\|x_{n+1}-x_\star\|^2 &\leq 
 \|x_{n}-x_\star\|^2 - \gamma^2 \| Y_\gamma \|^2 
 + 2 \gamma^2 \ps{\varphi(u),Y_\gamma} + 2 \gamma^2 \ps{\varphi(u),b} 
 \nonumber \\ 
&\phantom{\leq} + \gamma^2 \| b \|^2 
  - 2 \gamma \ps{\varphi(u) + \psi(u), x_n - x_\star} \nonumber \\ 
&\leq 
\|x_{n}-x_\star\|^2 - \gamma^2 (1 - \beta^{-1}) \| Y_\gamma \|^2 
 + \gamma^2( 1 + \beta^{-1}) \| b \|^2 \nonumber \\
&\phantom{\leq} + 2 \gamma^2 \beta \| \varphi(u) \|^2 
- 2 \gamma \ps{\varphi(u) + \psi(u), x_n - x_\star}\, , 
\label{stab} 
\end{align} 
where we used the inequality 
$|\ps{a,b}| \leq (\beta/2) \| a \|^2 + \| b\|^2 / (2\beta)$, where $\beta > 0$ 
is arbitrary. By Assumption~\ref{Bquad}, 
\[
\EE_n\| b \|^2\leq C(1+\|x_n\|^2)\leq 2C(1+\|x_\star\|^2+\|x_n-x_\star\|^2)
\]
for some (other) constant $C$.
Moreover $\EE_n \ps{\varphi(u) + \psi(u), x_n - x_\star} = 0$. Thus, 
\begin{multline*} 
\EE_n \|x_{n+1}-x_\star\|^2 \leq 
(1 + C \gamma_{n+1}^2) \|x_{n}-x_\star\|^2 \\ 
- \gamma_{n+1}^2 (1 - \beta^{-1}) 
             \int \| Y_{\gamma_{n+1}}(\xi, x_n) \|^2 \mu(d\xi) 
+ C \gamma_{n+1}^2 .
\end{multline*} 
Choose $\beta > 1$. Using the Robbins-Siegmund 
Lemma~\cite{robbins1971convergence} along with $(\gamma_n)\in \ell^2$, the 
conclusion follows.
\end{proof} 

\begin{remark}
This proposition calls for some comments.  In the standard forward-backward
algorithm described in the introduction of this paper, the operators $\sA$ and
$\sB$ are both deterministic, and $\sB$ is a single-valued operator satisfying 
a so-called cocoercivity property. In these conditions, the
iteration~\eqref{fb-deter} belongs to the class of the so-called
\emph{Krasnosel'ski\u{\i}-Mann} iterations, provided the fixed step size
$\gamma$ is chosen small enough~\cite{bau-com-livre11}.  A well known property
of these iterations is that the sequence $(x_n)$ is \emph{Fej\'er monotone}
with respect to $Z(\sA + \sB)$. Specifically, for all 
$x_\star \in Z(\sA + \sB)$, $(\| x_n - x_\star \|)$ is decreasing. 
In our situation, the forward operators $B(\xi, \cdot)$ are not required to be
single-valued. On the other hand, Assumptions~\ref{mom} and~\ref{Bquad} are
needed along with the fact that $(\gamma_n) \in \ell^2$.  Instead of the
Fej\'er monotonicity, we obtain the weaker result given by
Proposition~\ref{opial1}-\ref{randfejer}.  
\end{remark} 

The following lemma provides a moment control over the iterates $x_n$. 
\begin{lemma}
\label{2pm} 
Let Assumptions \emph{\ref{mom}} and \emph{\ref{Bquad}} in the statement of 
Theorem~\emph{\ref{th-apt}} hold true. 
Then, $\sup_n \EE \| x_n \|^{2p} < \infty$. 
\end{lemma}
\begin{proof}
We shall establish the result by recurrence over $p$. Proposition~\ref{opial1} 
shows that it holds for $p=1$. Assume that it holds for $p-1$. 
Using Assumption~\ref{mom}, choose 
$\varphi\in {\mathcal S}_{A(\,.\,,x_\star)}^{2p}$ and 
$\psi \in {\mathcal S}_{B(\,.\,,x_\star)}^{2p}$ such that 
$0= \int (\varphi + \psi) d\mu$.
 Inequality~\eqref{stab} shows that  for some constant $C>0$,
\begin{multline*}
  \|x_{n+1}-x_\star\|^2    
\leq  \|x_{n}-x_\star\|^2
-2\gamma_{n+1}\ps{\varphi(u_{n+1})+\psi(u_{n+1}),x_n-x_\star} \\
+C\gamma_{n+1}^2 ( \|\varphi(u_{n+1})\|^2+\|b(u_{n+1},x_n)\|^2)\,.
\end{multline*}
Raising both sides to the power $p$ then taking their expectations, we obtain 
\begin{equation} 
\EE \| x_{n+1} - x_\star \|^{2p} \leq 
\sum_{k_1 + k_2 + k_3 = p} 
% \binom{p}{k_1,k_2,k_3} 
\frac{p !}{k_1 ! k_2 ! k_3 ! } 
C^{k_2} (-2)^{k_3} \gamma_{n+1}^{2k_2 + k_3} T_n^{(k_1,k_2,k_3)} \, , 
% \times \\ 
% \EE \left(\| x_{n} - x_\star \|^{2k_1} 
% (\|\varphi(u_{n+1})\|^2+\|b(u_{n+1},x_n)\|^2)^{k_2} 
% \langle \varphi(u_{n+1})+\psi(u_{n+1}), x_n - x_\star  \rangle^{k_3}\right) . 
\label{expan} 
\end{equation} 
where we set for every $\vec{k}=(k_1,k_2,k_3)$,
\begin{multline*} 
T_n^{\vec{k}} = \EE \Bigl[ \| x_{n} - x_\star \|^{2k_1} \times 
(\|\varphi(u_{n+1})\|^2+\|b(u_{n+1},x_n)\|^2)^{k_2} \\
\times 
\langle \varphi(u_{n+1})+\psi(u_{n+1}), x_n - x_\star  \rangle^{k_3}\Bigr] \,.
\end{multline*} 
We can make the following observations: 
\begin{itemize}
\item By choosing $k_2 = k_3 = 0$, we observe that 
$\EE \| x_{n+1} - x_\star \|^{2p}$ is no greater than $\EE \| x_{n} - x_\star \|^{2p}$ 
plus some additional terms involving only smaller powers of $\| x_{n} - x_\star \|$. 

\item The term corresponding to $(k_1,k_2, k_3) = (p-1, 0,1)$ is zero since
$u_{n+1}$ and $\sigma(u_1,\ldots,u_n)$ are independent and 
$\EE_n\langle \varphi(u_{n+1})+\psi(u_{n+1}), x_n - x_\star  \rangle = 0$. This implies that any term in the 
sum except $\EE \| x_{n} - x_\star \|^{2p}$ is multiplied by $\gamma_{n+1}$, 
raised to a power greater than $2$. 

\item Consider the case $(k_1,k_2, k_3) \neq (p-1, 0,1)$ and 
$(k_1,k_2, k_3) \neq (p, 0,0)$. Using Jensen's inequality and the inequality
$x^k y^\ell \leq x^{k+\ell} + y^{k+\ell}$ for non-negative $x,y,k$ and $\ell$,
we get 
\begin{align*}
|T_n^{\vec{k}}| &\leq  \EE \Bigl[\| x_{n} - x_\star \|^{2k_1+k_3} \times 
   (\|\varphi(u_{n+1})\|^2+\|b(u_{n+1},x_n)\|^2)^{k_2}  \\
& 
\ \ \ \ \ \ \ \ \ \ \ \ \ \ \ \ \ \ \ \ \ \ \ \ \ \ \ \ \ \ \ \ 
\ \ \ \ \ \ \ \ \ \ \ \ \ \ \ \ 
    \times \| \varphi(u_{n+1})+\psi(u_{n+1})\|^{k_3}\Bigr] \\
&\leq  C \EE \Bigl[\| x_{n} - x_\star \|^{2k_1+k_3} \times 
      (\|\varphi(u_{n+1})\|^{2 k_2}+\|b(u_{n+1},x_n)\|^{2k_2} ) \\
& 
\ \ \ \ \ \ \ \ \ \ \ \ \ \ \ \ \ \ \ \ \ \ \ \ \ \ \ \ \ \ \ \ 
\ \ \ \ \ \ \ \ \ \ \ \ \ \ \ \ 
  \times ( \| \varphi(u_{n+1}) \|^{k_3} + \| \psi(u_{n+1})\|^{k_3}) \Bigr] \\
&\leq  
C \EE \Bigl[\| x_{n} - x_\star \|^{2k_1+k_3} 
                         \|b(u_{n+1},x_n)\|^{2k_2+k_3}\Bigr] \\
&\phantom{\leq} + 
  C \EE \Bigl[\| x_{n} - x_\star \|^{2k_1+k_3} \Bigr] 
    \EE\Bigl[ \| \varphi(u_{n+1}) \|^{2 k_2 + k_3} 
                         + \| \psi(u_{n+1})\|^{2 k_2 + k_3} \Bigr] . 
\end{align*} 
By conditioning on $\sigma(u_1,\ldots,u_n)$ and by using 
Assumption~\ref{Bquad}, we get 
\begin{multline*} 
\EE \Bigl[\| x_{n} - x_\star \|^{2k_1+k_3} 
          \|b(u_{n+1},x_n)\|^{2k_2+k_3}\Bigr] \\
\leq C \EE \Bigl[\| x_{n} - x_\star \|^{2k_1+k_3} 
                      ( 1 + \| x_n \|^{2k_2+k_3}) \Bigr] 
\leq C ( \EE \| x_{n} - x_\star \|^{2p} + 1) . 
\end{multline*} 
Noting that $2k_1+k_3\leq 2(p-1)$, we get that 
$\EE \| x_{n} - x_\star \|^{2k_1+k_3} < C $ by the induction hypothesis. Since 
$2k_2+k_3\leq 2p$ and since $\varphi$ and $\psi$ are $2p$-integrable 
selections, it follows that $|T_n^{\vec{k}}| \leq C(1+\EE\|x_n-x_\star\|^{2p})$.
Note also that in the considered case, one has $2k_2+k_3\geq 2$, which implies
that all terms $T_n^{\vec k}$ are multiplied by $\gamma_{n+1}^2$.
\end{itemize}

In conclusion, we obtain that 
\[
\EE \| x_{n+1} - x_\star \|^{2p} \leq 
\EE (1+C\gamma_{n+1}^2)\| x_{n} - x_\star \|^{2p} + C \gamma_{n+1}^2
\]
for some constant $C>0$. Starting from $n=0$ and iterating, we obtain that 
$\sup_n \EE \| x_{n} - x_\star \|^{2p} < \infty$.
\end{proof}

We now need to control the distances to $\mathcal D$ of the iterates $x_n$. Let
us start with an easy technical result, whose proof is left to the reader. 
\begin{lemma}
\label{p4}
For any $\varepsilon > 0$, there exist $C(\varepsilon) > 0$ and 
$C'(\varepsilon) > 0$ such that for any vectors $x, y \in \RR^N$, 
\[
\| x + y \|^2 \leq (1 + \varepsilon) \| x \|^2 + C(\varepsilon) \| y \|^2 , 
\quad \text{and} \quad 
\| x + y \|^4 \leq (1 + \varepsilon) \| x \|^4 + C'(\varepsilon) \| y \|^4 .
\]
\end{lemma}
% \begin{proof}
% Observe that 
% $| \langle x,y \rangle | \leq (\beta/2) \| x \|^2
% + \| y \|^2 / (2\beta)$ is true for any $\beta > 0$. Taking 
% $\beta=\varepsilon$, we obtain the first inequality. We also have 
% \begin{align*} 
% \| x + y \|^4 &\leq ( (1+\beta) \| x \|^2 + (1+1/\beta) \| y \|^2 ) ^2 \\
% &= (1+\beta)^2 \| x \|^4 + (1+1/\beta)^2 \| y \|^4 
% + 2 (1+\beta)(1+1/\beta) \| x \|^2 \| y \|^2  \\ 
% &\leq (1+\beta)^3 \| x \|^4 + (1 + 1/\beta)^3 \| y \|^4 . 
% \end{align*} 
% By choosing $\beta$ small enough, we obtain the second inequality. 
% \end{proof}

\begin{proposition} 
\label{d->0}
Let Assumptions \emph{\ref{mom}}, \emph{\ref{dist}}, \emph{\ref{PI-J}}, and 
\emph{\ref{Bquad}} of Theorem~ \emph{\ref{th-apt}} hold true. Then, 
$\bs d(x_n)$ tends a.s.~to zero.
Moreover, for every $\omega$ in a probability one set, there exists
$c(\omega)>0$ and a positive sequence $(c_m(\omega))_{m\in \NN}$ converging to
zero such that for every integer $n$ and every integer $m$ such that $n\geq m$,
\[
\sum_{k=m}^n \frac{\bs d(x_k)^2}{\gamma_k} 
\leq c_{m}(\omega) + c(\omega) \sum_{k=m}^n \gamma_k \,.
\]
\end{proposition} 

\begin{proof}
We start by writing 
$x_{n+1} = \Pi(u_{n+1}, x_n) + \gamma_{n+1} \delta_{n+1}$, where
\[
\delta_{n+1} = 
\frac{J_{\gamma_{n+1}}(u_{n+1}, x_n-\gamma_{n+1} b(u_{n+1},x_n)) - \Pi(u_{n+1}, x_n)}{\gamma_{n+1}} .
\]
Upon noting that $J_\gamma(\xi,\,.\,)$ is non-expansive for every $\xi$,
$$
\|\delta_{n+1} \|\leq 
\|b(u_{n+1},x_n) \|+ 
\frac{\|J_{\gamma_{n+1}}(u_{n+1}, x_n) - \Pi(u_{n+1}, x_n)\|}{\gamma_{n+1}} .
$$
Using Assumptions~\ref{PI-J} and~\ref{Bquad}, we have 
\begin{align*} 
\EE_n \| \delta_{n+1} \|^4 &= 4 \int \!\|b(\xi,x_n) \|^4\mu(d\xi)+
4\gamma_{n+1}^{-4} 
\int \!\| J_{\gamma_{n+1}}(\xi, x_n) - \Pi(\xi, x_n) \|^4 
\mu(d\xi) \\
&\leq C ( 1 + \| x_n \|^{2p} ) .
\end{align*} 
Therefore, by Proposition~\ref{opial1}-1., there exists a non-negative 
$c_1(\omega)$, which is a.s.~finite and satisfies 
$\EE_{n}\| \delta_{n+1} \|^4 \leq c_1(\omega)$ almost surely. By 
Lemma~\ref{2pm}, it also holds that $\sup_n \EE \| \delta_n \|^4 < \infty$. 
\\
Consider an arbitrary point $u \in \clos({\mathcal D})$. For any 
$\varepsilon > 0$, by Lemma~\ref{p4}, we have 
\[
\| x_{n+1} - u \|^2 \leq (1+\varepsilon) \| \Pi(u_{n+1}, x_n) - u \|^2
+ \gamma_{n+1}^2 C \| \delta_{n+1} \|^2 . 
\]
Since $\Pi(u_{n+1}, \cdot)$ is firmly non-expansive as the projector
onto a closed and convex set, we have 
\begin{align*}
\| \Pi(u_{n+1}, x_n) - u \|^2 
% &= 
%           \| \Pi(u_{n+1}, x_n) - x_n + x_n - u \|^2  \\
% &= \| \Pi(u_{n+1}, x_n) - x_n \|^2 + \| x_n - u \|^2  \\
% &\phantom{=} 
% -2 \langle (I - \Pi(u_{n+1}, \cdot)) x_n - (I - \Pi(u_{n+1}, \cdot)) u, 
%    x_n - u \rangle \\
&\leq \| x_n - u \|^2 - \| \Pi(u_{n+1}, x_n) - x_n \|^2 . 
\end{align*}
Taking $u = \bs\Pi(x_n)$, we obtain 
\begin{align*}
\bs d(x_{n+1})^2 &\leq \| x_{n+1} - \bs\Pi(x_n) \|^2 \\
&\leq (1+\varepsilon) (\bs d(x_n)^2 - d(u_{n+1}, x_n)^2) 
     + C \gamma_{n+1}^2 \| \delta_{n+1} \|^2 . 
\end{align*} 
Taking the conditional expectation $\EE_n$ at both sides of this inequality, 
using Assumption~\ref{dist} and choosing $\varepsilon$ small enough, we obtain 
the inequality $\EE_n \bs d^2(x_{n+1}) \leq \rho \bs d^2(x_n) 
+ \gamma_{n+1}^2 C \EE_n \| \delta_{n+1} \|^2$, where $\rho \in [0,1[$. 
It implies that $\bs d^2(x_n)$ tends to zero by the Robbins-Siegmund 
Theorem~\cite{robbins1971convergence}. Moreover, setting 
$\Delta_n = \bs d(x_n)^2 / \gamma_n$ and using the fact that 
$\gamma_n/\gamma_{n+1} \to 1$, we obtain that 
\[
\EE_n \Delta_{n+1} \leq \rho \Delta_n 
+ \gamma_{n+1} C \EE_n \| \delta_{n+1} \|^2
\]
for $n$ larger than some $n_0$. \\ 
By Lemma~\ref{p4} and the firm non-expansiveness of $\Pi(u_{n+1}, \cdot)$, 
we also have  
\begin{align}
\| x_{n+1} - u \|^4 &\leq (1+\varepsilon) \| \Pi(u_{n+1}, x_n) - u \|^4
+ \gamma_{n+1}^4 C \| \delta_{n+1} \|^4 \nonumber \\ 
&\leq (1+\varepsilon) 
( \| x_n - u \|^2 - \| \Pi(u_{n+1}, x_n) - x_n \|^2 )^2 
+ \gamma_{n+1}^4 C \| \delta_{n+1} \|^4  . 
\label{rv4}
\end{align} 
We also set $u = \bs\Pi(x_n)$ and apply the operator $\EE_n$ at both sides
of this inequality. By Assumption~\ref{dist}, we have  
\begin{align*} 
\int (\bs d(x)^2 - d(\xi, x)^2)^2 \mu(d\xi) 
&= \bs d(x)^4 + \int \!d(\xi, x)^4 \mu(d\xi) 
- 2 \bs d(x)^2 \! \int \! d(\xi, x)^2 \mu(d\xi) \\
&\leq \bs d(x)^4 - \bs d(x)^2 \int d(\xi, x)^2 \mu(d\xi) 
\leq (1-C) \bs d(x)^4 
\end{align*} 
since $d(\xi,x) \leq \bs d(x)$. Integrating~(\ref{rv4}), we obtain
\[
\EE_n \bs d^4(x_{n+1}) \leq \rho \bs d^4(x_n) + 
  \gamma_{n+1}^4 C \EE_n \| \delta_{n+1} \|^4\, , 
\]
where $\rho \in [0,1[$, hence
$\EE_n \Delta_{n+1}^2 \leq \rho \Delta_n^2  
+ \gamma_n^2 C \EE_n \| \delta_{n+1} \|^4 
$
for $n$ larger than some $n_0$. Taking the expectation at each side, 
iterating, and using the boundedness of $(\EE \| \delta_n \|^4)$, we obtain that 
$\EE\Delta_n^2 \leq C(\rho^n + \sum_{k=1}^n \gamma_k^2 \rho^{n-k} )$. 
Therefore,
\[
\sum_{n=0}^\infty \EE\Delta_n^2 \leq 
C\Bigl( 1 + \sum_{n=0}^\infty \gamma_n^2 \Bigr) < \infty. 
\]
Consequently, $\Delta_n \to 0$ almost surely. Moreover, the martingale 
\[
Y_n = \sum_{k=1}^n ( \Delta_k - \EE_{k-1} \Delta_k) 
\]
converges almost surely and in ${\mathcal L}^2(\Omega, \mcF, \PP; \RR)$. 
Letting $D_m^n = \sum_{k=m+1}^n \Delta_k$, where $m$ and $n$ are any two 
integers such that $0 < m < n$, we can write 
\begin{align*} 
D_m^n &= \sum_{k=m+1}^n \EE_{k-1} \Delta_k + Y_n - Y_m \\
&\leq \rho \sum_{k=m}^{n-1} 
  ( \Delta_k + C \gamma_{k+1} \EE_{k}\| \delta_{k+1} \|^2) + Y_n - Y_m \\
&\leq \rho \Delta_m + \rho D_m^n + 
\rho C \sqrt{c_1(\omega)} \sum_{k=m+1}^n \gamma_k + Y_n - Y_m  . 
\end{align*}
To conclude, we have 
\[
D_m^n \leq \frac{\rho}{1-\rho} \Delta_m + \frac{Y_n - Y_m}{1-\rho} 
+ \frac{\rho C \sqrt{c_1(\omega)}}{1-\rho} \sum_{k=m+1}^n \gamma_k \,.
\]
Since $\Delta_m \to 0$, and since $(Y_n(\omega))_{n\in \NN}$ is almost surely 
a Cauchy sequence, we obtain the desired result. 
\end{proof}

\begin{lemma}
\label{lem:hbd}
Let Assumptions~\emph{\ref{cpct}} and~\emph{\ref{Bquad}} hold true. For any 
compact set $K$, there exists a constant $C>0$ and $\varepsilon\in ]0,1]$
such that for all $x\in K$ and all $\gamma>0$,
$$
\|h_\gamma(x)\|\leq C+ 2\frac{\bs d(x)}\gamma \, , 
$$
and moreover, 
$$
\int ( \|Y_\gamma(\xi,x)\|^2+\|b(\xi, x)\|^2)^{\frac{1+\varepsilon}2}\mu(d\xi) \leq C \left[1+\left(\frac{\bs d(x)}\gamma\right)^{1+\varepsilon}\right]\,.
$$
\end{lemma}
\begin{proof}
  Set $x\in K$, and introduce some $\tilde x\in \mD$ such that 
$\|x-\tilde x\|\leq 2\bs d(x)$.
Relying on the fact that $A_\gamma(\xi,\,.\,)$ is $\frac 1\gamma$-Lipschitz 
continuous, 
\begin{align*}
  \|Y_\gamma(\xi,x)\|&\leq\|A_\gamma (\xi,\tilde x)\|+\frac 1\gamma   \|x-\gamma b(\xi,x)-\tilde x\| \\
&\leq\|A_0 (\xi,\tilde x)\|+\|b(\xi,x)\|+ 2\frac{\bs d(x)}\gamma\,.
\end{align*}
Therefore,
$$
\|h_\gamma(x)\|\leq \int\|A_0 (\xi,\tilde x)\|\mu(d\xi) +
2\int \|b(\xi,x)\|\mu(d\xi) +  2\frac{\bs d(x)}\gamma\,.
$$
The first two terms are independent of $\gamma$ and, by Assumptions~\ref{cpct} and~\ref{Bquad}, 
are bounded functions of $x$ on the compact $K$. This proves the first statement of the Lemma.
Let $\varepsilon=\varepsilon(K)$ be the exponent defined in Assumption~\ref{cpct}. There exists a constant $C$ 
such that
\begin{align*}
( \|Y_\gamma(\xi,x)\|^2 & +\|b(\xi, x)\|^2)^{\frac{1+\varepsilon}2} \\ 
&\leq C(\|Y_\gamma(\xi,x)\|^{1+\varepsilon}+\|b(\xi, x)\|^{1+\varepsilon})\\
& \leq  C\Bigl(
\Bigl(\|A_0 (\xi,\tilde x)\|+\|b(\xi,x)\|+ 2\frac{\bs d(x)}\gamma
       \Bigr)^{1+\varepsilon}+\|b(\xi, x)\|^{1+\varepsilon}\Bigr)\\
& \leq C'\Bigl(2^{\varepsilon} \|A_0 (\xi,\tilde x)\|^{1+\varepsilon}
 +2^{1+2\varepsilon} \|b(\xi,x)\|^{1+\varepsilon}
 + 2^{1+3\varepsilon}\Bigl(\frac{\bs d(x)}\gamma\Bigr)^{1+\varepsilon}\Bigr).
\end{align*}
By Assumption~\ref{Bquad} and since $\int \|b(\xi,x)\|^{1+\varepsilon}\mu(d\xi)\leq 1+ \int \|b(\xi,x)\|^{2}\mu(d\xi)$,
there exists some (other) constant $C$ such that
\begin{multline*} 
\int ( \|Y_\gamma(\xi,x)\|^2+\|b(\xi, x)\|^2)^{\frac{1+\varepsilon}2}
\mu(d\xi) \\
\leq C\Bigl( \int \|A_0 (\xi,\tilde x)\|^{1+\varepsilon}\mu(d\xi) + 1+\|x\|^2 
  + \Bigl(\frac{\bs d(x)}\gamma\Bigr)^{1+\varepsilon}\Bigr).
\end{multline*} 
The proof is concluded using Assumption~\ref{cpct}.
\end{proof}

\subsubsection*{End of the Proof of Theorem~\ref{th-apt}} 

Recall \eqref{eq-int}. Given an arbitrary real number $T > 0$, we shall study
the asymptotic behavior of the family of functions $\{x(\tau_n +
\cdot)\}_{n\in\NN}$ on the compact interval $[0,T]$. 

Given $\delta > 0$, we have $\| H(t+\delta) - H(t) \| \leq \int_t^{t+\delta}\|h_{\gamma_{r(s)+1}} (x_{r(s)})\|ds$.
By Proposition~\ref{opial1}-1, the sequence $(x_n)$ is bounded a.s. Thus, by Lemma~\ref{lem:hbd},
 there exists a constant $c_1=c_1(\omega)$ such that for almost every $\omega$,
\begin{align*} 
\| H(t+\delta) - H(t) \| &\leq c_1\delta + 2 \int_t^{t+\delta} \frac{\bs d(x_{r(s)})}{\gamma_{r(s)+1}} ds \\
&\leq c_1\delta +  \int_t^{t+\delta} \left(1+\frac{\bs d(x_{r(s)})^2}{\gamma_{r(s)+1}^2}\right) ds \\
&= (c_1+1)\delta +  \int_t^{t+\delta} \frac{\bs d(x_{r(s)})^2}{\gamma_{r(s)+1}^2}ds 
\\
&\leq (c_1 + c_2 + 1) \delta + e(t) 
% \label{eq-H} 
\end{align*}
for some $e(t) \to_{t\to\infty} 0$, where the last inequality is due to Proposition~\ref{d->0}.
We also observe from Proposition~\ref{opial1} and Assumption~\ref{Bquad} that 
$M_n$ is a martingale in ${\mathcal L}^2(\Omega, \mcF, \PP; \RR^N)$, that 
\[
\EE \| M_n \|^2 \leq 
\EE\Bigl[ 2\sum_{k=1}^\infty \gamma_k^2 \int \| Y_{\gamma_k}(\xi, x_k) \|^2 
\mu(d\xi)  + 2\sum_{k=1}^\infty \gamma_k^2 \int \| b(\xi, x_k) \|^2 
\mu(d\xi) \Bigr] \, , 
\]
and that the right-hand side is finite. 
Hence, $M_n$ converges almost surely. Therefore, on a probability one set,
the family of continuous time processes 
$( M(\tau_n + \cdot) - M(\tau_n) )_{n\in\NN}$ converges to zero uniformly on 
$\RR_+$. The consequence of these observations is that 
on a probability one set, the family of processes $\{ z_n(\,.\,) \}_{n\in\NN}$, 
where $z_n(t) = x(\tau_n + t)$,  is equicontinuous. 
Specifically, for each $\varepsilon > 0$, there exists $\delta > 0$ such that
\[
\limsup_n \sup_{0\leq t,s\leq T, |t-s|\leq \delta} \| z_n(t) - z_n(s) \|
\leq \varepsilon .
\]
This family is moreover bounded by Proposition~\ref{opial1}-1. By the
Arzel\`a-Ascoli theorem, it has an accumulation point for the uniform
convergence on $[0,T]$, for an arbitrary $T>0$. From any sequence of integers,
we can extract a subsequence (which we still denote as $(z_n)$ with slight
abuse), and a continuous function $z(\cdot)$ on $[0, T]$, such that $(z_{n})$
converges to $z$ uniformly on $[0, T]$. Hence, for $t\in[0, T]$,
\begin{align*} 
z(t) - z(0) &= - \lim_{n\to\infty} \int_{0}^{t} h_{\gamma_{r(\tau_n+s)+1}}(x_{r(\tau_n+s)}) \, ds \\
&= - \lim_{n\to\infty} \int_{0}^{t} ds \int_{\Xi} \mu(d\xi) \, (g_n^{(a)}(\xi,s) +g_n^{(b)}(\xi,s)) \, , 
\end{align*}
where we set
$g_n^{(a)}(\xi,t) \eqdef Y_{\gamma_{r(\tau_n+s)+1}}(\xi, x_{r(\tau_n+s)})$ 
and $g_n^{(b)}(\xi,t) \eqdef b(\xi, x_{r(\tau_n+s)})$. 
Define the mapping $g_n\eqdef(g_n^{(a)},g_n^{(b)})$ on 
$\Xi\times [0,T]\to {\mathbb R}^{2N}$.
Recalling that the sequence $(\tilde x_n)$ belongs to a 
compact set, say $K$, let $\varepsilon \in ]0,1]$ be the exponent defined in 
Lemma~\ref{lem:hbd}. By the same Lemma,
\begin{align*}
  \int_{0}^{T} ds \int_{\Xi} \mu(d\xi) \,
  \|g_n(\xi,s)\|^{1+\varepsilon} 
&\leq c  \Bigl[T
  +\int_{0}^{T}\Bigl(\frac{\bs d(x_{r(\tau_n+s)})}{\gamma_{r(\tau_n+s)+1}}
   \Bigr)^{1+\varepsilon}ds\Bigr] \\
&\leq c  \Bigl[T
  +T^{\frac{1-\varepsilon}2}\Bigl(\int_{0}^{T}\frac{\bs d(x_{r(\tau_n+s)})^2}{\gamma_{r(\tau_n+s)+1}^2}ds\Bigr)^{\frac{1+\varepsilon}2}\Bigr]\\
&\leq c_1
\end{align*}
for some constants $c$ and $c_1$. 
Therefore, the sequence of  functions $(g_n)$ is bounded in 
${\mathcal L}^{1+\varepsilon} (\Xi \times [0,T], 
{\mcT} \otimes {\mcB}([0,T]), \mu \otimes \lambda; \RR^{2N} )$,
where $\lambda$ is the Lebesgue measure on $[0, T]$.
The statement extends to the sequence of functions
\[
\bigl( G_n(\xi,t)
 =(g_n(\xi,t), \|g_n^{(a)}(\xi,t)\| , \|g_n^{(b)}(\xi,t)\| ) \bigr)_n, 
\]
which is uniformly bounded in 
${\mathcal L}^{1+\varepsilon} (\Xi \times [0,T], 
 {\mcT} \otimes {\mcB}([0,T]), \mu \otimes \lambda; \RR^{2N+2} )$.
We can extract from this sequence a  subsequence that converges weakly in this Banach space to a function 
$F:\Xi \times [0,T] \to \RR^{2N+2}$. We decompose $F$ as $F(\xi,t) = (f(\xi,t), \kappa(\xi,t),\upsilon(\xi,t))$, 
where $\kappa,\upsilon$ are real-valued, and where 
$f(\xi,t) = (f^{(a)}(\xi,t),f^{(b)}(\xi,t))$
with $f^{(a)}, f^{(b)}:\Xi \times [0,T] \to \RR^{N}$. 
Using the weak convergence $(g_n^{(a)},g_n^{(b)})\rightharpoonup (f^{(a)},f^{(b)})$, 
we obtain 
$$
z(t) - z(0) = - \int_{0}^{t} ds \left(\int_{\Xi}  f^{(a)}(\xi,s)\mu(d\xi)+ \int_{\Xi} f^{(b)}(\xi,s)  \mu(d\xi)\right).
$$
It remains to prove that for almost every $t\in [0,T]$,  $f^{(a)}(\,.\,,t)\in A(\,.\,,z(t))$ and 
$f^{(b)}(\,.\,,t)\in B(\,.\,,z(t))$ $\mu$-almost everywhere, along with $z(0)\in \clos(\mD)$. 
This shows that indeed $z(t) = \Phi(z(0), t)$ for every $t\in [0,T]$, and it
follows that $x(t)$ is a.s.~an APT of the differential inclusion~\eqref{di}. 

By Mazur's theorem, there exists a function $J : \NN \to \NN$ and 
a sequence of sets of weights 
$(\{ \alpha_{k,n}, k=n\ldots, J(n) \, : \, \alpha_{k,n} \geq 0, 
\sum_{k=n}^{J(n)} \alpha_{k,n} = 1 \} )_n$ such that  the sequence of functions defined by
$$
\bar G_n(\xi,s) = \sum_{k=n}^{J(n)} \alpha_{k,n}\, G_k(\xi,s)
$$ 
converges strongly to $F$. In the same way, we define  
$\bar g_n(\xi,s) \eqdef \sum_k\alpha_{k,n}\, g_k(\xi,s)$, and similarly for 
$\bar g_n^{(a)},\bar g_n^{(b)}$.
Extracting a further subsequence, we obtain the 
$\mu\otimes \lambda$-almost everywhere convergence of $\bar G_n$ to $F$.
By Fubini's theorem, for almost every $t\in [0,T]$, there exists a $\mu$-negligible set
such that for every $\xi$ outside this set, $\bar G_n(\xi,t)\to F(\xi,t)$.
From now on to the end of this proof, we fix such a $t\in [0,T]$.

As $\bs d(x_n)\to 0$, $z(t)\in \clos(\mD)$ 
(this holds in particular when $t=0$, hence $z(0)\in \clos(\mD)$). 
Following the same arguments as in the proof of Proposition~\ref{max}, 
it holds that $z(t)\in \clos(D(\xi))$ for all $\xi$ outside a $\mu$-negligible set.

Define $\eta_n(\xi) \eqdef 
J_{\gamma_{m+1}}(\xi,x_m-\gamma_{m+1} b(\xi,x_m))-z(t)+\gamma_{m+1} b(\xi,x_m)$ with $m=r(\tau_n+t)$.
Using the same approach as in the proof of Proposition~\ref{max}, it can be shown that, as $n\to\infty$, 
$\eta_n(\,.\,)$ tends to zero almost surely along a subsequence. 
We now consider an arbitrary $\xi$ outside a $\mu$-negligible set, 
such that $\eta_n(\xi)\to 0$ and $z(t)\in \clos(D(\xi))$.

Let $(u,v)$ be an arbitrary element of $A(\xi,\cdot)$. By the monotonicity of $A(\xi,\cdot)$,
$$
\ps{v-Y_\gamma(\xi,x),u-J_\gamma(\xi,x-\gamma b(\xi,x))}\geq 0
\qquad (\forall x\in \RN,\gamma>0) \, , 
$$
and we obtain
\begin{align*}
&  \ps{v-\bar g_n^{(a)}(\xi,t),u-z(t)}  = \sum_{k=n}^{J(n)} \alpha_{k,n}\, \ps{v- g_k^{(a)}(\xi,t),u-z(t)} \\
& \geq \sum_{k=n}^{J(n)} \alpha_{k,n}\, \ps{v- g_k^{(a)}(\xi,t),\eta_k(\xi,t)-\gamma_{r(\tau_k+t)+1}b(\xi,x_{r(\tau_k+t)})}\\
&\geq -\Bigl(\|v\| +\sum_{k=n}^{J(n)} \alpha_{k,n}\, \|g_k^{(a)}(\xi,t)\|\Bigr)
\sup_{k\geq n} \left(\|\eta_k(\xi,t)\|\!+\!\gamma_{r(\tau_k+t)+1}\|b(\xi,x_{r(\tau_k+t)})\|\right).
\end{align*}
The term enclosed in the first parenthesis of the above right-hand side converges to $\|v\|+\kappa(\xi,t)$, 
while the supremum converges to zero using Assumption~\ref{Bquad}.
As $\bar g_n^{(a)}(\xi,t)\to f^{(a)}(\xi,t)$, it follows that
$$
\ps{v-f^{(a)}(\xi,t),u-z(t)}\geq 0 \, , 
$$
and by the maximality of $A(\xi,\cdot)$, it holds that $f^{(a)}(\xi,t)\in A(\xi,z(t))$.
The proof that $f^{(b)}(\xi,t)\in B(\xi,z(t))$ follows the same lines.
\hfill\qed 

\subsection{Proof of Corollary~\ref{avg}} 

The proof is based on the study of the family of empirical measures of a 
process close to $x(t)$. Using \cite{ben-sch-00}, we show that any 
accumulation point of this family is an invariant measure for the flow $\Phi$. 
The corollary is then obtained by showing that the mean of such an invariant 
measure belongs to $\mZ$. 

Let ${\bs x}_n = {\bs\Pi}(x_n)$ be the projection of $x_n$ on $\clos(\mD)$, 
and write 
\[
\bar {\bs x}_n = \frac{\sum_{k=1}^n \gamma_k {\bs x}_k}{\sum_{k=1}^n \gamma_k} .
\]
Let ${\bs x}(\omega, t)$ be the $\Omega \times \RR_+ \to \RR^N$ process 
obtained from the piecewise constant interpolation of the sequence 
$({\bs x}_n)$, 
namely ${\bs x}(\omega, t) = {\bs x}_n$ for $t\in [\tau_n, \tau_{n+1} [$. 
On $(\Omega, \mcF, \PP)$, let $(\mcF_t)$ be the filtration generated by 
the process obtained from the similar piecewise constant interpolation of 
$(u_n)$. 
With regard to this filtration, $\bs x$ is progressively measurable. 
It is moreover obvious that ${\bs x}(\omega, \cdot)$ is an APT for~\eqref{di} 
for almost all values of $\omega$. 
Let $\{ \nu_t(\omega, \cdot) \}_{t\geq 0}$ be the family of empirical measures 
of ${\bs x}(\omega, \cdot)$. Observe from Theorem~\ref{th-apt} that for almost 
all $\omega$, there is a compact set $K(\omega)$ such that the support 
$\support(\nu_t(\omega, \cdot))$ is included in $K(\omega)$ for all 
$t\geq 0$, which shows that the family $\{ \nu_t(\omega, \cdot) \}_{t\geq 0}$ 
is tight. Hence this family has accumulation points. Let $\nu$ be the weak
limit of $(\nu_{t_n})$ along some sequence $(t_n)$ of times. 
By \cite[Th.~1]{ben-sch-00}, $\nu$ is invariant for the flow $\Phi$. 
Clearly, $\support(\nu)$ is a compact subset of $\clos(\mD)$. Moreover, 
for any $x \in \support(\nu)$ and any $t \geq 0$, 
$\Phi(x, t) \in \support(\nu)$. Indeed, suppose for the sake of contradiction
that there exists $t_0 > 0$ such that $\Phi(x, t_0) \not\in \support(\nu)$. 
Then, $\Phi(B(x, \varepsilon) \cap \clos(\mD), t_0) \subset 
\support(\nu)^{\text{c}}$ for some $\varepsilon > 0$ 
by the continuity of $\Phi$ and the closedness of $\support(\nu)$, where 
$B(x,\varepsilon)$ is the closed ball with centre $x$ and radius $\varepsilon$.
Since $\nu( \Phi(B(x, \varepsilon) \cap \clos(\mD), 0) ) > 0$,
we obtain a contradiction.  
We also know from \cite{bai-bre-76} or \cite[Th.~5.3]{pey-sor-10} that there 
exists $\varphi : \clos(\mD) \to \mZ$ such that 
\[
\forall x\in \clos(\mD), \quad 
\frac 1t \int_0^t \Phi(x,s) \, ds 
\xrightarrow[t\to\infty]{} \varphi(x) .
\]
By the dominated convergence and Fubini's theorems, we now have 
\begin{align*} 
\int \varphi(x) \, \nu(dx) &= 
  \int \nu(dx) \lim_{t\to\infty}  \frac 1t \int_0^t ds \, \Phi(x,s) 
= \lim_{t\to\infty}  \frac 1t \int_0^t ds \int \nu(dx)  \, \Phi(x,s) \\
&= \int x \, \nu(dx) \, , 
\end{align*} 
which shows that $\int x \, \nu(dx) \in \mZ$ by the convexity of this set. 
Since we have $\int x \, d\nu_{t_n} \to \int x \, d\nu$ as $n\to\infty$, 
we conclude
that all the accumulation points of $(\bar{\bs x}_n)$ belong to $\mZ$. 
On the other hand, since 
${\mathcal R}_{2p}(x_\star) \neq \emptyset$ for each $x_\star \in \mZ$, a 
straightforward inspection of the proof of Proposition~\ref{opial1}-3.~shows that 
$(\|x_n - x_\star \|)$ converges almost surely for each $x_\star \in \mZ$. 
From these two facts, we obtain by \cite{pas-79} or \cite[Lm 4.2]{pey-sor-10}
that $(\bar{\bs x}_n)$ converges a.s.~to a point of $\mZ$. Since 
$x_n - {\bs x}_n \to 0$ a.s., the convergence of $(\bar x_n)$ to the same 
point follows. 
\hfill\qed 

\subsection{Proof of Corollary~\ref{1/2p}} 
\label{prf-1/2p} 

Let us start with a preliminary lemma.
\begin{lemma}
\label{lem:zerICT}
  Let $\sA\in {\mathcal M}$ be demipositive. Assume that the set $\zer(\sA)$
of zeros of $\sA$ is not empty. Let 
$\Psi:\clos(\dom(\sA))\times \RR_+\to \clos(\dom(\sA))$
be the semiflow associated to the differential inclusion $\dot z(t)\in -\sA(z(t))$. 
Then, any ICT set of $\Psi$ is included in $\zer(\sA)$. 
\end{lemma}
\begin{proof}
  Let $K$ be an ICT set and let $U$ be an arbitrary, bounded and open set of 
$\RN$ such that $K\cap U\neq\emptyset$. Define 
$G_t\eqdef \bigcup_{s\geq t}\Psi(U,s)$ for all $t\geq 0$. For any 
$x_*\in \zer(\sA)$ and any $x\in U$, 
\[
\|\Psi(x,t)\|\leq \|\Psi(x,t)-\Psi(x_*,t)\| + \|x_*\|\leq \|x-x_*\|+\|x_*\|\, .
\]
Therefore, $G_0$ is a bounded set. By~\cite[Prop.~3.10]{ben-hof-sor-05}, the 
set $G = \bigcap_{t\geq 0}\clos({G_t})$ is an attractor for $\Psi$ with a fundamental neighbourhood $U$.
As $K\cap U\neq \emptyset$, it follows that $K\subset G$ by \cite[Corollary 5.4]{ben-(cours)99}. We finally check that $G\subset \zer(\sA)$.
Let $y\in G$, that is, $y=\lim_{k\to\infty}\Psi(x_k,t_k)$ for some sequence $(x_k,t_k)$ such that $x_k\in U$ and $t_k\to\infty$.
By compactness of $\clos(U)$, the sequence $x_k$ can be chosen such that $x_k\to \bar x$ for some $\bar x\in \clos(U)$.
Therefore, $y=\lim_{k\to\infty}\Psi(\bar x,t_k)$,
which by demipositivity of $\sA$, implies 
$y\in \zer(\sA)$ \cite{bru-75,pey-sor-10}.
\end{proof}

By Theorem~\ref{th-apt} and the discussion of Section~\ref{evolution}, $L(x)$ is an ICT set.
Using Lemma~\ref{lem:zerICT} and the standing hypotheses, $L(x)\subset\mZ$.
On the other hand, since ${\mathcal R}_{2}(x_\star)\neq \emptyset$ 
for all $x_* \in \mZ$, a straightforward inspection of the proof of
Proposition~\ref{opial1}-3.~shows that $\|x_n - x_* \|$ converges almost  
surely for any of those $x_*$. 
By Opial's lemma \cite[Lm 4.1]{pey-sor-10}, 
we obtain the almost sure convergence of $(x_n)$ to a point of $\mZ$. 
\hfill\qed 

\subsection{Proof of Corollary~\ref{coro:optim}}

Define the probability distribution 
$\zeta \eqdef \sum_{i=0}^m \alpha_i \delta_i$ on $\{0, 1, \ldots, m \}$.

On the space $\sX \times \{0,\ldots, m \}$ 
equipped with the probability $\mu = \nu \otimes \zeta$, let 
$\xi = (\eta, i)$, and define the random operators  $A$ and $B$ by 
\begin{equation*}
\label{eq:ABoptim}
A(\xi,\cdot) \eqdef \left\{\begin{array}{ll} 
      \alpha_0^{-1} \partial_x g(\eta, \cdot), & \text{if} \ i = 0, \\
      N_{\mathcal C_i}, & \text{otherwise}, \end{array}\right. 
\quad \text{and} \quad 
B(\xi,\cdot) \eqdef \partial_x f(\eta, \cdot) .
\end{equation*}
The Aumann integral ${\mathcal B}(x) = \int \partial f(\eta,x)d\pi(\eta)$ 
coincides with $\partial F(x)$ by~\cite{roc-wet-82} (see also the discussion 
in Section~\ref{A=dG}). Similarly, 
${\mathcal A}(x) = \partial (G(x) + \iota_{\mathcal C})(x)$. 
The operator $\mA$ is thus maximal. It
holds that $\mA+\mB = \partial (F+G+\iota_{\mathcal C})$, which is maximal,
demipositive, and whose zeros coincide with the minimizers of $F+G$ over
${\mathcal C}$.  The end of the proof consists in checking the assumptions of 
Corollary~\ref{1/2p}. It follows the same line as~\cite{bia-(arxiv)15} 
and is left to the reader.
\hfill\qed 

\section{Perspectives} 
\label{sec-persp}

Beyond the forward-backward algorithm, the concept of random maximal monotone
operators can be used to study stochastic versions of other popular
optimization algorithms that rely on the monotone operator theory.  Our next 
research direction is therefore to extend our approach to other kinds of
algorithms, such as the Douglas-Rachford algorithm, as a way to construct new
families of stochastic approximation algorithms. In this perspective, the
present paper may contain useful ingredients.

It would also be interesting to weaken the assumption that the ``innovation''
$(u_n)$ is an iid sequence. More involved random models are often useful.
Among those are the ones where the innovation is a Markov chain controlled by
the iterates.  Such models are popular in the classical stochastic
approximation literature.

Another research direction includes the case where the step size of the
algorithm is constant. In this context, the APT property does not hold and the
iterates are no longer expected to converge a.s., due to the persistence of the
random effects. Tools from the weak convergence theory of stochastic processes
can be useful to address this setting.

Finally, we believe that our algorithm can be shown to be useful to address
several specific applications in the field of convex optimization and
variational inequalities. An important aspect is to instanciate the algorithm
in practical scenarios related to machine learning, signal processing, or game
theory.

\section{Conclusions}
\label{sec-conclu}

The question of providing stochastic versions of well-known
deterministic algorithms relying on maximal monotone operators has become
increasingly popular. In particular, several authors have studied the effects
of additive random errors on the behavior of the iterates, showing
that the errors have no effect on the limiting points, provided some
adequate vanishing condition of the former. The approach taken by 
this paper is conceptually different in the sense that the operators
themselves are assumed to be random.  This situation involves two
key-ingredients. The first one is the Aumann expectation of the random
operators.  The second one is the notion of asymptotic
pseudotrajectory, borrowed from Bena\"im and Hirsch, which is used to
relate the iterates to a continuous-time dynamical system.

\section*{Acknowledgements}
This work was partially funded by phi-TAB, the Orange - Telecom ParisTech 
think tank, and by the ASTRID program of the French Agence Nationale de la
Recherche (ODISSEE project ANR-13-ASTR-0030). 

% \bibliographystyle{spmpsci_unsrt}
% \bibliography{math}

\def\cprime{$'$} \def\cdprime{$''$} \def\cprime{$'$}

\end{document}